\documentclass[12pt,letterpaper]{article}

\usepackage{amssymb,amsfonts,amscd,amsthm}
\usepackage[all,arc]{xy}
\usepackage{enumerate, bbm}
\usepackage{mathrsfs}
\usepackage{tikz}
\usepackage[left=0.9in,top=0.9in,right=0.9in,bottom=0.9in]{geometry}
\usepackage{mathtools}
\usepackage{hyperref}
\usepackage{cleveref}
\usepackage{color}
\usepackage{graphicx}
\usepackage{enumitem} 
\graphicspath{ {TexImages/} }

\newtheorem{thm}{Theorem}[section]
\newtheorem*{thm*}{Theorem}

\newtheorem{cor}[thm]{Corollary}
\newtheorem{prop}[thm]{Proposition}
\newtheorem{lem}[thm]{Lemma}
\newtheorem{claim}[thm]{Claim}

\newtheorem{quest}[thm]{Question}
\newtheorem{prob}[thm]{Problem}

\newtheorem{conj}[thm]{Conjecture}

\theoremstyle{definition}
\newtheorem{defn}{Definition}

\hypersetup{
	colorlinks,
	citecolor=blue,
	filecolor=blue,
	linkcolor=blue,
	urlcolor=blue,
	linktocpage
}

\setlist[enumerate]{itemsep=2ex, topsep=2ex} 
\setlist[itemize]{itemsep=2ex, topsep=2ex}

\newcommand{\R}{\mathbb{R}}

\newcommand{\E}{\mathbb{E}}
\newcommand{\al}{\alpha}

\newcommand{\gam}{\gamma}

\newcommand{\ep}{\varepsilon}

\newcommand{\del}{\delta}
\newcommand{\Del}{\Delta}

\newcommand{\half}{\frac{1}{2}}
\newcommand{\quart}{\frac{1}{4}}

\newcommand{\sm}{\setminus}
\newcommand{\sub}{\subseteq}

\newcommand{\tr}[1]{\textrm{#1}}

\newcommand{\mr}[1]{\mathrm{#1}}

\renewcommand{\SS}[1]{\textcolor{red}{#1}}

\newcommand{\dist}{\mr{dist}}

\newcommand{\ty}{\mr{ty}}

\title{Generalized Quasikernels in Digraphs}
\author{Sam Spiro\footnote{Dept.\ of Mathematics, Rutgers University {\tt sas703@scarletmail.rutgers.edu}. This material is based upon work supported by the National Science Foundation Mathematical Sciences Postdoctoral Research Fellowship under Grant No. DMS-2202730.}}
\date{\today}

\begin{document}
	
	\maketitle
	
	\begin{abstract}
		Given a digraph $D$, we say that a set of vertices $Q\subseteq V(D)$ is a $q$-kernel if $Q$ is an independent set and if every vertex of $D$ can be reached from $Q$ by a path of length at most $q$.  In this paper, we initiate the study of several extremal problems for $q$-kernels.  For example, we introduce and make progress on (what turns out to be) a weak version of the Small Quasikernel Conjecture, namely that every digraph contains a $q$-kernel with $|N^+[Q]|\ge \frac{1}{2}|V(D)|$ for all $q\ge 2$.
	\end{abstract}

	\section{Introduction}
	This paper concerns problems about (finite) digraphs, for which we recall a few basic definitions.  Given vertices $u,v$ of a digraph $D$, we define $\dist(u,v)$ to be the length of a shortest directed path from $u$ to $v$, and for a set of vertices $S$, we define $\dist(S,v)=\min_{u\in S} \dist(u,v)$.   We say a set of vertices $S$ is \textit{independent} if there exist no arcs between two vertices of $S$.
	
	The following (less basic) objects  will be the starting point for this paper.
	
	\begin{defn}
		A set of vertices $K$ of a digraph $D$ is said to be a \textit{kernel} if (1) $K$ is an independent set and (2) 
		$\dist(K,x)\le 1$ for all $x\in V(D)$  (that is, every vertex $x\in V(D)$ is either contained in $K$ or has a vertex of $K$ as an in-neighbor).  
	\end{defn}
	\begin{defn}
		A set of vertices $Q$ of a digraph $D$ is said to be a \textit{quasikernel} if (1) $Q$ is an independent set and (2) $\dist(Q,x)\le 2$ for all $x\in V(D)$ (that is, every vertex $x\in V(D)$ can be reached from $Q$ by a directed path of length at most 2).
	\end{defn}
	
	The concept of kernels was introduced by Von Neumann and Morgenstern \cite{von1947theory} in relation to problems from game theory.   Not every digraph contains a kernel (for example, a directed 3-cycle does not), but it turns out that every digraph has a quasikernel, as was originally shown by Chv\'atal and Lov\'asz \cite{chvatal1974every}.  
	
	Because every digraph contains a quasikernel, it is natural to ask extremal type problems about the set of quasikernels for a given digraph.  One conjecture in this direction is the Small Quasikernel Conjecture due to P.L.\ Erd\H{o}s and Sz\'ekely \cite{erdHos2010two}, where here we recall a digraph is \textit{source-free} if it contains no vertices with in-degree 0. 
	
	\begin{conj}[Small Quasikernel Conjecture]\label{conj:smallquasi}
		If $D$ is a source-free digraph, then $D$ contains a quasikernel $Q$ with $|Q|\le\half |V(D)|$.
	\end{conj}
	The source-free condition is necessary for the conclusion $|Q|\le\half |V(D)|$ to hold, as can be seen by considering, for example, $D$ to be a digraph with no arcs.  The bound $\half |V(D)|$ is best possible, as can be seen by considering (disjoint unions of) directed 2-cycles and 4-cycles.

	The Small Quasikernel Conjecture seems difficult, with very little progress being made between its original statement in 1976 and 2020.  However, since 2020, there have been a number of results solving the Small Quasikernel in a few special cases \cite{ai2023results,langlois2022algorithmic,langlois2023quasi,van2021kernels}, with this recent increase in popularity largely due to work of Kostochka, Luo, and Shan \cite{kostochka2020towards} who (in particular) showed that 4-chromatic digraphs satisfy the Small Quasikernel Conjecture.  
	
	We provide our own small step towards the Small Quasikernel Conjecture  by proving what appears to be the first general bound for the size of the smallest quasikernel in source-free digraphs.  Here and throughout our logarithms are written base 2.
	\begin{thm}\label{thm:smallQuasi}
		If $D$ is a source-free digraph, then $D$ contains a quasikernel with \[|Q|\le |V(D)|-\max\{\lfloor \sqrt{|V(D)|}\rfloor,\quart \sqrt{|V(D)|\log |V(D)|}\}.\]
	\end{thm}
	Note that the bound of $|V(D)|-\lfloor \sqrt{|V(D)|}\rfloor$ from \Cref{thm:smallQuasi} is almost always weaker than the other bound $|V(D)|-\quart \sqrt{|V(D)|\log |V(D)|}$.  We  include this generally weaker bound both because it has a simple proof and because it is tight when $D$ is a directed 2-cyles or 4-cycle.
	
	Much more can be said about the Small Quasikernel Conjecture, and we refer the interested reader to the nice survey by P.L.\ Erd\H{o}s, Gy\H{o}ri, Mezei, Salia, and Tyomkyn \cite{erdHos2023small} for a more  thorough overview of this problem.

	\section{Main Results}
	We introduce a natural generalization of kernels and quasikernels which does not appear to have been considered within the context\footnote{Our definition of $q$-kernels are a special case of the more general notion of a $(k,\ell)$-kernel due to Kw\'asnik and Borowiecki~\cite{kwasnik1980k} which has been studied a fair amount.  However, the results for $(k,\ell)$-kernels are mostly focused on existence questions rather than extremal problems, and our particular definition of $q$-kernels does not appear to have been given any special attention in this setting.} of the Small Quasikernel Conjecture. 
	
	\begin{defn}
		Given an integer $q\ge 1$ and a digraph $D$, we say that a set of vertices $Q\sub V(D)$ is a \textit{$q$-kernel} if $Q$ is independent and if $\dist(Q,v)\le q$ for every vertex $v\in V(D)$. 
	\end{defn}
	For example, 1-kernels are kernels and 2-kernels are quasikernels.

	In view of the Small Quasikernel Conjecture, it is natural to ask how small of a $q$-kernel  one can find in source-free digraphs.  Unlike with quasikernels, this problem turns out to be quite easy for $q\ge 3$. Indeed, in \Cref{sec:disjoint} we will show the following.
	\begin{prop}\label{prop:disjoint3}
		If $D$ is a source-free digraph and $q\ge 3$, then $D$ contains two disjoint $q$-kernels.  In particular, $D$ contains a $q$-kernel of size at most $\half |V(D)|$.
	\end{prop}
	This result is best possible by considering  disjoint unions of directed 2-cycles.  This result does not hold for $q=2$, i.e.\ there exist source-free digraphs such that every two quasikernels intersect each other, as was originally proven by Gutin, Koh, Tay, and Yeo \cite{gutin2004number}.
	
	Although \Cref{prop:disjoint3} solves the most natural analog of \Cref{conj:smallquasi}, it also serves as the starting point for many new problems which we explore further below.  In particular, we study the more general problem of finding $r$ disjoint $q$-kernels in digraphs, as well as strengthening the bound of \Cref{prop:disjoint3} when extra structural information for $D$ is assumed.  In addition to this, we introduce (what turns out to be) a weak version of the Small Quasikernel Conjecture and make some progress towards this problem, especially in the setting of $q$-kernels.
	
	
	\subsection{Disjoint $q$-kernels}
	
	Here we study the problem of finding $r$ disjoint $q$-kernels in a digraph $D$.  One immediate obstacle is the following structure, where here and throughout
	we define the open in-neighborhood of a set $S$ by
	\[N^-_D(S):=\{u\in V(D)\sm S:\exists v\in S,\ uv\in E(D)\}.\]
	\begin{defn}
		Given a digraph $D$, we say that a set of vertices $S\sub V(D)$ is a \textit{source set} if $N^-(S)=\emptyset$.  We say that $S$ is an \textit{$s$-source set} if $S$ is non-empty and $|S|\le s$.
	\end{defn}
	For example, 1-source sets are just sources.
	
	Observe that if $D$ contains an $(r-1)$-source set $S$, then it does not contain $r$ disjoint $q$-kernels for any $q$.  Indeed, every $q$-kernel needs to use at least one vertex of $S$, and hence $D$ can contain at most $|S|<r$ disjoint $q$-kernels.  
	
	It is not difficult to show that $(r-1)$-source sets are the only obstruction for $D$ containing $r$ disjoint $q$-kernels for some $q\le |V(D)|$.  With more work,  one can significantly improve upon this trivial bound $q\le |V(D)|$.
	
	\begin{thm}\label{thm:disjointWeak}
		If a digraph has no $(r-1)$-source sets, then it contains $r$ disjoint $2^{r+1}$-kernels.
	\end{thm}
	In other words, if a digraph contains $r$ disjoint $q$-kernels for some $q$, then it in fact contains $r$ disjoint $2^{r+1}$-kernels. We will ultimately prove a somewhat stronger result which guarantees disjoint sets $Q_1,\ldots,Q_r$ such that each $Q_i$ is a $q_i$-kernel with $q_i\approx i\cdot 2^{r-i}$; see \Cref{cor:disjointFull} for more.

	\subsection{Smaller $q$-kernels}
	The examples showing sharpness of \Cref{conj:smallquasi} and \Cref{prop:disjoint3} both involve disjoint unions of  directed 2-cycles or 4-cycles.  As such, it is reasonable to ask if these bounds can be improved if $D$ avoids certain cycle lengths.  A reasonable guess here would be the following.
	
	\begin{quest}\label{quest:smallCycles}
		Let $q\ge 2$ be an integer and let $D$ be a source-free digraph with $L(D)$ the set of $\ell$ such that $D$ contains a directed $\ell$-cycle. Is it true that $D$ contains a $q$-kernel $Q$ such that
		\[|Q|\le \max_{\ell\in L(D)} \frac{\lceil \ell/(q+1) \rceil}{\ell}\cdot |V(D)|?\]
	\end{quest}
	Observe that this upper bound is tight if $D$ is the disjoint union of directed $\ell$-cycles.  Also observe that $D$ being source-free implies $L(D)\ne \emptyset$, so this maximum is well defined.
	
	We suspect \Cref{quest:smallCycles} is too ambitious to be true in general.  We are, however, able to verify a few cases of \Cref{quest:smallCycles} when our digraph is bipartite, which is analogous to the work of \cite{kostochka2020towards} solving the Small Quasikernel Conjecture for digraphs of chromatic number at most 4.  To this end, we say a partition $U,V$ of $V(D)$ is a \textit{bipartition} of  $D$ if every arc has exactly one vertex in $U$ and exactly one vertex in $V$, and we say that $D$ is \textit{bipartite} if it has a bipartition.  
	
	Note that for bipartite source-free digraphs, we trivially have  quasikernels with $|Q|\le \frac{1}{2}|V(D)|$ by simply considering $Q$ to be the smaller part of the bipartition, and this is tight whenever $D$ is the disjoint union of directed 2-cycles and 4-cycles.  We show that this is the unique extremal construction for bipartite digraphs.
	\begin{prop}\label{prop:smallQuasiBipartite}
		If $D$ is a source-free bipartite digraph which is not the disjoint union of directed 2-cycles and 4-cycles, then $D$ contains a quasikernel $Q$ with $|Q|<\half |V(D)|$.
	\end{prop}
	
	We do not know how to improve the bound of \Cref{prop:smallQuasiBipartite} if we impose the stronger condition that $D$ contain no directed 2-cycles or 4-cycles, and in particular we are quite far from the optimal possible bound of $|Q|\le \frac{2}{5}|V(D)|$ predicted by \Cref{quest:smallCycles}  (due to $\ell=10$).   However, if one shifts from quasikernels to the setting of $q$-kernels with $q>2$, then one can prove sharp results when the smallest cycle length is small relative to $q$.  Moreover, we can do this in the more general setting of digraphs which have kernels (which in particular includes source-free bipartite digraphs  by considering either part of the partition).

	\begin{thm}\label{thm:bipartite}
		Let $D$ be a source-free  digraph which has a kernel $K$.  If $q,\ell\ge 3$ are integers such that $\ell\le (q+3)/2$ and such that every directed cycle of $D$ of even length has length at least $\ell$, then $D$ contains a $q$-kernel $Q$ of size at most $\frac{1}{\ell}|V(D)|$ with $Q\sub K$.
	\end{thm}
	For example, taking $q=\ell=3$ implies that every source-free digraph with a kernel and no directed 2-cycles contains a 3-kernel of size at most $\frac{1}{3}|V(D)|$, which is best possible by considering directed 6-cycles.  More generally, the bound for $Q$ in \Cref{thm:bipartite} is best possible and matches what \Cref{quest:smallCycles} predicts if $D$ is the disjoint union of directed $\ell$-cycles when $\ell$ is even (the even condition guaranteeing $D$ has a kernel), or if $D$ is the disjoint union of directed $2\ell$-cycles when $2\ell\ge q+2$.
	
	We suspect that one can extend our approach for \Cref{thm:bipartite} to verify \Cref{quest:smallCycles} for bipartite digraphs with other cycle restrictions.  However, it is not possible to solve \Cref{quest:smallCycles} in full for bipartite digraphs if we continue imposing the additional condition $Q\sub K$ in \Cref{thm:bipartite} (which will turn out to be vital to our proof); see the arXiv only appendix for some concrete examples of this.


	\subsection{Large $q$-kernels}
	This subsection studies the problem of finding ``large'' $q$-kernels in digraphs. The most straightforward problem of trying to maximize $|Q|$ amongst $q$-kernels $Q$ is not interesting, as $|Q|$ could be as small as 1 if $D$ is a tournament.  However, this problem becomes interesting if we change how we measure  $Q$.  To this end, we denote the closed out-neighborhood of a set $S\sub V(D)$ by
	\[N^+[S]=S\sqcup N^+(S)=S\cup  \{v:\exists u\in S,\ uv\in E(D)\}.\]
	\begin{conj}[Large Quasikernel Conjecture]\label{conj:largeQuasi}
		Every digraph $D$ contains a quasikernel $Q$ with $|N^+[Q]|\ge \frac{1}{2}|V(D)|$.
	\end{conj}
	This bound is asymptotically best possible by considering $Q$ to be an Eulerian touranment.  The main motivation for this conjecture is the following.
	
	\begin{prop}\label{prop:small2Large}
		The Small Quasikernel Conjecture implies the Large Quasikernel Conjecture.
	\end{prop}
	In fact,  we will prove a slightly stronger result \Cref{prop:strongLargeSmall} which shows that proving any non-trivial linear upper bound on the size of a smallest quasikernel implies a non-trivial linear lower bound on $|N^+[Q]|$.  With this, we see that being able to prove a linear lower bound for $|N^+[Q]|$ is a barrier towards improving our \Cref{thm:smallQuasi} into a linear upper bound.

	We are able to take a modest step towards  the Large Quasikernel Conjecture, with us making significantly more progress for 3-kernels.
	
	\begin{thm}\label{thm:large}
		Let $D$ be a digraph.
		\begin{itemize}
			\item[(a)] There exists a quasikernel $Q$ with $|N^+[Q]|\ge |V(D)|^{1/3}$.
			\item[(b)] There exists a 3-kernel $Q$ with $|N^+[Q]|\ge \frac{1}{3} |V(D)|$.
		\end{itemize}
	\end{thm}
	\subsection{Organization and Conventions}
	Our proofs are broken into four essentially independent sections organized by increasing levels of technical detail: we begin with our general upper bound for quasikernels in \Cref{sec:quasi}, our results about large $q$-kernels in \Cref{sec:large}, our results about disjoint $q$-kernels in \Cref{sec:disjoint}, and  our results about bipartite digraphs in \Cref{sec:bipartite}.  We close the paper with a large number of open problems in \Cref{sec:con}.
	

	Throughout the text, we always assume that our digraphs are finite.  Our proofs will sometimes use without comment the fact that every digraph contains a quasikernel. We let $C_\ell$ denote the directed $\ell$-cycle, i.e. the digraph with vertices $v_1,\ldots,v_\ell$ and arcs $v_iv_{i+1}$ for all $1\le i\le \ell$ (with these indices written cyclically). 
	\section{Proof of \Cref{thm:smallQuasi}: Small Quasikernels}\label{sec:quasi}
	We begin with a simple lemma that will be used in the next section as well.
	\begin{lem}\label{lem:X}
		For every digraph $D$ and $x\in V(D)$, there exists a quasikernel $Q$ of $D$ such that $Q\cap N^+(x)=\emptyset$ and $Q\cap N^-[x]\ne\emptyset$.
	\end{lem}
	\begin{proof}
		Let $D'=D-N^+[x]$ and let $Q'$ be any quasikernel of $D'$, noting that $Q'\cap N^+(x)=\emptyset$ by construction.  If $Q'\cap N^-(x)\ne\emptyset$, then it is straightforward to show that $Q'$ is a quasikernel for $D$ which intersects $ N^-[x]$. If instead $Q'\cap N^-(x)=\emptyset$, then $Q=Q'\cup \{x\}$ is an independent set, and it is not difficult to see that $Q$ is a quasikernel which satisfies $Q\cap N^+(x)=\emptyset$ and $Q\cap N^{-}[x]\ne \emptyset$, giving the result.
	\end{proof}

	We also need the following  observation.
	\begin{lem}\label{lem:ind}
		If $D$ is a digraph with minimum in-degree $\del$ and maximum out-degree $\Del$, then every independent set $I$ of $D$ satisfies \[|I|\le |V(D)|-\frac{\del}{\del+\Del}|V(D)|.\]
	\end{lem}
	\begin{proof}
		Let $I$ be an independent set and let $D'\sub D$ be the subdigraph consisting only of the arcs that go from $V(D)\sm I$ to $I$. Because $I$ is independent, each $v\in I$ has $\deg_{D'}^-(v)=\deg_D^-(v)\ge \del$.  From this we conclude
		\[\del |I|\le \sum_{v\in I} \deg_{D'}^-(v)=\sum_{u\in V(D)\sm I} \deg_{D'}^+(u)\le  \Del (|V(D)|-|I|).\]
		Rearranging gives the  result.
	\end{proof}
	These two results give our weaker bound on the size of small quasikernels.
	\begin{prop}\label{prop:weakSmallQuasi}
		If $D$ is a source-free digraph, then $D$ contains a quasikernel $Q$ with \[|Q|\le |V(D)|-\lfloor \sqrt{|V(D)|}\rfloor.\]
	\end{prop}
	\begin{proof}
		First consider the case that $D$ contains a vertex $x$ with out-degree at least $\lfloor  \sqrt{|V(D)|} \rfloor$.  By \Cref{lem:X}, there exists a quasikernel $Q$ disjoint from $N^+(x)$, so \[|Q|\le |V(D)|-|N^+(x)|\le |V(D)|-\lfloor  \sqrt{|V(D)|}\rfloor,\] proving the bound.  
		
		We now assume $D$ has maximum out-degree at most $\lfloor  \sqrt{|V(D)|}\rfloor-1$, noting that the source-free condition is equivalent to saying that $D$ has minimum in-degree at least 1.  By \Cref{lem:ind}, any independent set (and in particular, any quasikerenl) has size at most 
		\[|V(D)|-\frac{1}{\lfloor  \sqrt{|V(D)|}\rfloor} |V(D)|\le |V(D)|- \sqrt{|V(D)|},\]
		proving the result.
	\end{proof}
	As an aside, this same approach can be used to show that digraphs with minimum in-degree $\delta$ contain quasikernels of size at most $|V(D)|-\sqrt{\del |V(D)|}+\delta$; see also the discuss around \Cref{quest:minDegree} for problems related to this.
	
	To improve \Cref{prop:weakSmallQuasi}, we extend the argument of \Cref{lem:X} by deleting a set $A$ of vertices rather than just a single vertex.  To this end, we say a set $A\sub V(D)$ is \textit{kernel-perfect} if for every $A'\sub A$, the digraph $D[A']$ contains a kernel.    
	\begin{lem}\label{lem:kernelPerfect}
		If $D$ is a digraph and $A\sub V(D)$ is kernel-perfect, then $D$ contains a quasikernel which is disjoint from $N^+(A)$.
	\end{lem}
	We note that this lemma follows from a more general recent result \cite[Lemma 3.1]{ai2024variable}, but for completeness we include the easy proof.
	\begin{proof}
		Let $D'=D-N^+[A]$ and let $Q'$ be any quasikernel of $D'$, noting that $Q'\cap N^+(A)=\emptyset$ by construction.  Let $A'=A\sm N^+(Q')$, and by hypothesis there exists some $K\sub A'$ which is a kernel of $D[A']$.  Define $Q=Q'\cup K$, noting that this is an independent set since $Q'$ is an independent set disjoint from $N^+[A]\supseteq N^+[K]$ and $K\sub A'$ is an independent set disjoint from $N^+[Q']$.  It is straightforward to show that every vertex can be reached from $Q$ by a path of length at most 2, proving the result.
	\end{proof}
	Our goal will be to use \Cref{lem:kernelPerfect} to show that if $D$ is a source-free digraph, then either every independent set has size at most $|V(D)|-m$ or there exists a kernel-perfect set $A$ such that $|N^+(A)|\ge m$ for some $m\approx \sqrt{|V(D)|\log |V(D)|}$.  This bound on $m$ is the best one can obtain with this approach.  Indeed, if one takes $T$ to be a random tournament on $\sqrt{n\log n}$ vertices and then forms a digraph $D$ from $T$ by adding $\sqrt{n/\log n}$ leaves to each vertex of $T$, then $|V(D)|\approx n$, its largest independent set $D-T$ is around size $|V(D)|-\sqrt{|V(D)|\log |V(D)|}$, and the best kernel-perfect sets $A$ one can use are transitive subtournaments of $T$ which all have size about $\log |V(D)|$.
	
	Partially motivated by this extremal example for our approach, we observe the following.
	\begin{lem}\label{lem:largeKernelPerfect}
		If $D$ is a non-empty digraph, then there exists a set of vertices $A\sub V(D)$ which is kernel-perfect for $D$ with $|A|\ge \log |V(D)|$.
	\end{lem}	
	\begin{proof}
		First observe that if $D$ contains a directed two-cycle $uv,vu$ and if $A$ is kernel-perfect for the digraph  $D-uv$, then $A$ is still kernel-perfect for $D$ (since any kernel of $D'[A']$ is still independent and a kernel in $D[A']$).  As such, it suffices to find a large kernel-perfect subset of the digraph $D'\sub D$ obtained by arbitrarily deleting one arc from each directed 2-cycle of $D$.
		
		Let $D''\supseteq D'$ be an arbitrary tournament on the same vertex set of $D'$ obtained by arbitrarily adding arcs between vertices $u,v$ which have no arcs between them (noting that this $D''$ will be a tournament since $D$ contains no directed 2-cycles).  A well known result (see \cite[Exercise 10.44]{lovasz2007combinatorial}) implies there exists a set of vertices $A\sub V(D'')$ of size at least $1+\lfloor \log |V(D'')|\rfloor\ge \log|V(D)|$ such that $D''[A]$ is a transitive tournament.  This implies that for any $A'\sub A$, the digraph $D'[A']$ contains no directed cycles, and it is known that such digraphs always contain (unique) kernels (see \cite[Exercise 8.5]{lovasz2007combinatorial}).  We conclude that $A$ is a kernel-perfect subset of $D'$ of the desired size, proving the result.
	\end{proof}
	We now have all we need to prove our main result about small quasikernels.
	\begin{proof}[Proof of \Cref{thm:smallQuasi}]
		Recall that we wish to prove that every source-free digraph $D$ has a quaikernel $Q$ with 
		\[|Q|\le |V(D)|-\max\{\lfloor \sqrt{|V(D)|}\rfloor,\quart \sqrt{|V(D)|\log |V(D)|}\}.\] 
		In view of \Cref{prop:weakSmallQuasi}, we only need to prove the bound $|Q|\le |V(D)|-\quart \sqrt{|V(D)|\log |V(D)|}$, and for ease of notation we let $n=|V(D)|$ and $m=\quart \sqrt{n\log n}$.  If every independent set of $D$ has size at most $n-m$ then we are done, so we can assume there exists an independent set $I$ of size at least $n-m$.  Our aim now is to show the following.
		\begin{claim}
			There exists a kernel-perfect subset $A\sub V(D)$ such that $|N^+(A)|\ge m$.
		\end{claim}
		\begin{proof}
			The result is trivial if there exists a vertex $u$ with $\deg^+(u)\ge m$, so we may assume this is not the case.
			
			We will prove the claim by finding such an $A$ which is disjoint from $I$ with the additional property that $|N^+(A)\cap I|\ge m$.  To this end, define the leftover vertices $L=V(D)\sm I$.  Because $D$ is source-free and $I$ is an independent set, there exists some function $f:I\to L$ such that $f(v)\in N^-(v)$ for every $v\in I$ (e.g.\ by defining $f(v)$ to be the smallest element of $N^-(v)$ under some arbitrary ordering of $V(D)$).  For each $u\in L$, define its weight $\phi(u)=|f^{-1}(u)|$, and for any given set $U\sub L$ define $\phi(U)=\sum_{u\in U}\phi(u)$.  Observe from the definitions that for any $A\sub L$ we have
			\[|N^+(A)\cap I|\ge \phi(A),\]
			so it will suffice to find a kernel-perfect $A$ with $\phi(A)\ge m$.  Towards this end, we observe from the definitions that
			\[\phi(L)=|I|\ge n-m\]\[|L|=n-|I|\le m.\] 
	
			Let $L'\sub L$ consist of the vertices $u$ with $\phi(u)\le \half n m^{-1}$ and observe that
			\[\phi(L\sm L')=\phi(L)-\phi(L')\ge n-m-|L|\cdot \half n m^{-1}\ge \half n-m.\]
			For each $i\ge 0$ define the set \[L_i=\{u\in L:2^{i-1}n m^{-1}<\phi(u)\le 2^i n m^{-1}\},\] noting that these $L_i$ sets partition the vertices of $L\sm L'$.
				
			We claim that $L_i=\emptyset$ for any $i$ with $2^{i-1}\ge  n^{-1}m^2$.  Indeed, any $u$ in such an $L_i$ would have \[|N^+(u)|\ge \phi(u)> 2^{i-1}n m^{-1}\ge m,\]
			a contradiction to our assumption at the start of the claim.  We also claim that some $i\ge 0$ has $\phi(L_i)\ge 2^{-1-i}(\half n-m)$, as otherwise we would have
			\[\phi(L\sm L')=\sum_{i\ge 0} \phi(L_i)<\left(\half n-m\right)\sum_{i\ge 0} 2^{-1-i}=\half n-m,\]
			contradicting our bound above.
			
			In total, these subclaims imply the existence of some $i$ such that $\phi(L_i)\ge 2^{-1-i}(\half n-m)$ and such that $2^{i}<2n^{-1}m^2$.  By definition of $L_i$ this implies
			\[2^{-1-i}\left(\half n-m\right)\le \phi(L_i)\le |L_i|\cdot 2^{i-1}n m^{-1},\]
			and rearranging gives
			\[|L_i|\ge 2^{-2i}\left(\half m-n^{-1}m^2\right)>\frac{1}{8} n^2 m^{-3}-\quart nm^{-2}\ge \frac{1}{16} n^2 m^{-3} \ge n^{1/4},\]
			where the second to last inequality used $n^2 m^{-3}\ge 4 nm^{-2}$ (equivalently that $n\ge 4 m$, i.e.\ that $n\ge \log n$), and the last inequality used $\frac{1}{16} n^2 m^{-3}= 4 n^{1/2} (\log n)^{-3/2}$ and $4  n^{1/4}\ge  (\log n)^{-3/2}$ for $n\ge 2$ (which must hold if $D$ is source-free). By \Cref{lem:largeKernelPerfect}, there exists some $A\sub L_i$ such that $A$ is a kernel-perfect subset of $D[L_i]$ (and hence of $D$) and such that $|A|\ge \log(|L_i|)\ge \quart \log(n)$.  By definition of $A\sub L_i \sub L\sm L'$ this implies
			\[|N^+(A)|\ge \phi(A)\ge |A|\cdot  \half n m^{-1}\ge \frac{1}{8} n \log(n) m^{-1}=\half \sqrt{n \log n}\ge m,\]
			proving the claim.
		\end{proof}
		The result follows from this claim together with \Cref{lem:kernelPerfect}.
	\end{proof}

	\section{Proof of \Cref{thm:large}: Large $q$-kernels}\label{sec:large}
	We begin by proving a slight strengthening of \Cref{prop:small2Large}, namely that weak variants of the Small Quasikernel imply weak variants of the Large Quasikernel Conjecture.
	
	\begin{prop}\label{prop:strongLargeSmall}
		If $0<\ep\le 1$ is a real number such that every source-free digraph $D'$ contains a quasikernel $Q'$ with $|Q'|\le (1-\ep)|V(D')|$, then every digraph $D$ contains a quasikernel $Q$ with $|N^+[Q]|\ge \ep |V(D)|$.
	\end{prop}
	Note that the case $\ep=1/2$ is exactly the statement of \Cref{prop:small2Large}.
	\begin{proof}[Proof of \Cref{prop:small2Large}]
		Let $D$ be a minimal counterexample to the statement for some given value of $\ep$.  If $D$ contains a source $v$, then  $D$ being a minimal counterexample implies we can find a quasikernel $Q'$ in $D'=D-N^+_D[v]$ with \[|N^+_{D'}[Q']|\ge \ep (|V(D)|-|N^+_D[v]|).\]  One can then verify that $Q:=Q'\cup \{v\}$ is a quasikernel of $D$ and that \[|N^+_D[Q]|= |N^+_{D'}[Q']|+|N^+_D[v]|\ge \ep |V(D)|+(1-\ep)|N^+_D[v]|\ge \ep |V(D)|,\] giving the result.  Thus we may assume $D$ is source-free.  
		
		Let $k$ be a large integer to be determined later and define a digraph $D'$ by taking each vertex $v\in V(D)$ and adding $k$ new vertices $v_1,\ldots,v_k$ together with the arcs $vv_i$ for all $i$.  Let $Q'$ be a smallest quasikernel of $D'$, noting that $|Q'|\le  (1-\ep)|V(D')|=(1-\ep)(k+1)|V(D)|$ by hypothesis. 
		
		It is not difficult to verify that $Q:=Q'\cap V(D)$ is a quasikernel of $D$ and that $Q'\sm Q=Q'\sm V(D)$ contains every vertex $v_i$ with $v\in V(D)\sm N^+_D[Q]$ (as otherwise $Q'$ would not be a quasikernel for $D'$).  These observations imply
		\[(1-\ep)(k+1)|V(D)|\ge |Q'|\ge |Q'\sm Q|\ge k(|V(D)|-|N^+_D[Q]|),\]
		and rearranging  gives
		\[|N^+[Q]|\ge \ep |V(D)|-\frac{1-\ep}{k}|V(D)|.\]
		Since $|N^+[Q]|$ is an integer, this implies $|N^+[Q]|\ge \lceil \ep|V(D)|-\frac{1-\ep}{k}|V(D)|\rceil$ for all $k\ge 1$.  By taking $k$ sufficiently large, we see that this ceiling is at least $\ep |V(D)|$, proving the result.
	\end{proof}

	We now prove that digraphs have $q$-kernels with $|N^+[Q]|$ relatively large.  The following observations will be useful, where here we say an independent set $S$ is \textit{maximal} if $S\cup \{v\}$ is not independent for all $v\notin S$ and that it is \textit{acyclic} if it contains no directed cycles.
	
	\begin{lem}\label{lem:greedy}
		Let $D$ be a digraph.
		\begin{itemize}
			\item[(a)] If every vertex of $D$ has out-degree at most $\Del$, then there exists a set $A\sub V(D)$ such that $D[A]$ is acyclic with $|A|\ge \frac{|V(D)|}{\Del+1}$.
			\item[(b)] There exists a maximal independent set $S$ with $|N^+[S]|\ge \half |V(D)|$.
		\end{itemize}
	\end{lem}
	As an aside, \Cref{lem:greedy}(b) can be viewed as a very weak version of the Large Quasikernel Conjecture (as the Large Quasikernel Conjecture requires finding such a set which is additionally a quasikernel).
	\begin{proof}
		For (a), let $A$ be a largest set such that $D[A]$ is acyclic, noting that at least one such set exists by considering $A=\emptyset$.  If $|A|< \frac{|V(D)|}{\Del+1}$, then there exists some vertex $v\notin N^+[A]$ (since $N^+[A]$ has size at most $|A|(\Del+1)<|V(D)|$). Observe that $D[A\cup \{v\}]$ continues to be acyclic (since $v$ can not be a vertex in any directed cycle with $A$ by construction), contradicting the maximality of $|A|$ and giving the result.
		
		For (b), we introduce the notation $N[S]=N^+[S]\cup N^-(S)$, i.e.\ this is the set of vertices which are contained in some arc containing a vertex of $S$.  Let $S$ be a largest independent set satisfying  \begin{equation}|N^+[S]|\ge \half |N[S]|,\label{eq:S}\end{equation} noting that at least one such independent set exists by considering $S=\emptyset$.  
		
		Let $T=V(D)\sm N[S]$.  If $T=\emptyset$ (i.e. if $N[S]=V(D)$) then $S$ is a maximal independent set with $|N^+[S]|\ge \half |V(D)|$, so we may assume for contradiction that this is not the case. Let $D'=D[T]$.  Because $D'$ is a non-empty digraph, there exists some vertex $v\in T$ with \begin{equation}|N^+_{D'}[v]|\ge \half |N_{D'}[v]|,\label{eq:v}\end{equation} 
		namely this holds by taking any $v\in T$ with at least as many out-neighbors as in-neighbors.  In this case, $S'=S\cup \{v\}$ is a larger independent set satisfying
		\[N_D[S']=N_D[S]\cup (N_{D}[v]\sm N_D[S])=N_D[S]\cup N_{D'}[v],\]
		and
		\[N_D^+[S']=N_D^+[S]\cup (N^+_D[v]\sm N_D^+[S])\supseteq N_D^+[S]\cup (N_{D}^+[v]\sm N_D[S])= N_D^+[S]\cup N_{D'}^+[v].\]
		Combining these with \eqref{eq:S} and \eqref{eq:v} gives \[|N^+_D[S']|\ge |N^+_D[S]|+|N^+_{D'}[v]|\ge  \half |N_D[S]|+  \half |N_{D'}[v]|= \half |N_D[S']|,\]
		a contradiction to the maximality of $|S|$.  We conclude that a maximal independent set with the desired properties must exist.
	\end{proof}
	
	We next show that if $D$ contains a small quasikernel, then $D$ contains a large 3-kernel.
	
	\begin{prop}\label{prop:quasiVs3}
		If $D$ is a digraph and $Q\sub V(D)$ is a quasikernel with $|N^+[Q]|=\ep |V(D)|$ for some $0<\ep<1$, then there exists a 3-kernel $Q'\sub V(D)$ with $|N^+[Q']|\ge \frac{1-\ep}{2}|V(D)|$.
	\end{prop}
	\begin{proof}
		Let $T=V(D)\sm N^+(Q)$, noting that $|T|\ge (1-\ep)|V(D)|$ and that $Q\sub T$.  By \Cref{lem:greedy}(b), the digraph $D[T]$ contains a  maximal independent set $Q'$ with
		\[|N^+_D[Q']|\ge |N^+_{D[T]}[Q']|\ge \half |T|\ge \frac{1-\ep}{2}|V(D)|.\]
		It thus remains to show that $Q'$ is a 3-kernel.  Note that $Q'$ is an independent set by construction.  For any $w\in V(D)$, we know there exists some $v\in Q$ such that $\dist(v,w) \leq 2$.  Observe that $v\in Q$ has $N^+(v)\sub N^+(Q)$ since $Q$ is an independent set, and hence $v$ has no out-neighbors in $T=V(D)\sm N^+(Q)$.  As such, since $Q'$ is a maximal independent set of $D[T]$ and $v\in Q\sub T$, either $v\in Q'$ or some in-neighbor of $v$ lies in $Q'$.  In either case we find $\dist(Q',w)\le 1+\dist(v,w)\le 3$.  This shows $Q'$ is a 3-kernel, proving the result.
	\end{proof}
	
	\Cref{prop:quasiVs3} is essentially a refinement of \Cref{thm:large}(b), and in particular implies this result.
	
	\begin{proof}[Proof of \Cref{thm:large}(b)]
		Recall that we aim to show every digraph $D$ contains a 3-kernel with $|N^+[Q]|\ge \frac{1}{3}|V(D)|$.  Let $Q\sub V(D)$ be any quasikernel.  If $|N^+[Q]|\ge \frac{1}{3}|V(D)|$ then we are done, otherwise there exists a 3-kernel of the desired size by \Cref{prop:quasiVs3}.
	\end{proof}
	It remains to show that digraphs always contain large quasikernels.  For this we use the following.

	\begin{lem}\label{lem:acyclic}
		If $D$ is a digraph and $A\sub V(D)$ is such that $D[A]$ is acyclic, then there exists a quasikernel $Q\sub V(D)$ such that $A\sub N^+[Q]$.
	\end{lem}
	\begin{proof}
		Let $S\supseteq A$ be a maximal subset such that $D[S]$ is acyclic.  Observe that every $v\notin S$ has an in-neighbor in $S$, as otherwise $S\cup \{v\}$ would be a larger subset inducing an acyclic digraph (since $v$ can not be a part of any directed cycle in $S\cup \{v\}$).
		
		It is known that every acyclic digraph contains a (unique) kernel (see \cite[Exercise 8.5]{lovasz2007combinatorial}), and we let $Q$ denote the kernel of $D[S]$.  Observe that $A\sub S\sub N^+[Q]$ by definition of $Q$ being a kernel of $D[S]$.  We also have $\dist(Q,v)\le 2$ for all $v\notin S$, since by the observation above, $v$ has an in-neighbor in $S$ which is at distance at most 1 from $Q$.   Thus $Q$ is a quasikernel with the desired properties, giving the result.
	\end{proof}
	We now have all we need to prove our main result for this section.
	\begin{proof}[Proof of \Cref{thm:large}(a)]
		Recall that we aim to show every digraph contains a qusaikernel with $|N^+[Q]|\ge |V(D)|^{1/3}$.  Let $X\sub V(D)$ be the set of vertices $x$ which are not contained in any quasikernel of $D$.
		
		First consider the case $|X|\ge |V(D)|-|V(D)|^{2/3}$.  By the second half of \Cref{lem:X}, we in particular see that for each $x\in X$, there exists some $y\in V(D)\sm X$ such that $yx\in E(D)$ (namely, by considering the quasikernel $Q$ which intersects $N^+[x]$ after noting that $Q$ is disjoint from $X$ by definition of $X$).  This implies that there are at least $|X|$ arcs from $V(D)\sm X$ to $X$, and hence there must exist some $y\in V(D)\sm X$ with 
		\[\deg^+(y)\ge \frac{|X|}{|V(D)|-|X|}\ge |V(D)|^{1/3}-1.\]
		Because $y\notin X$, there exists some quasikernel $Q$ such that $y\in Q$, and hence $|N^+[Q]|\ge \deg^+(y)+1\ge |V(D)|^{1/3}$, proving the result.
		
		Next consider $|X|\le |V(D)|-|V(D)|^{2/3}$ and let $D'=D-X$, noting that $|V(D')|\ge |V(D)|^{2/3}$.  By the same argument as above, if there exists $y\in V(D')$ with $\deg^+_{D'}(y)\ge |V(D)|^{1/3}-1$, then any quasikernel of $D$ containing $y$ gives the desired result, so we may assume $\deg^+_{D'}(y)\le |V(D)|^{1/3}-1:=\Del$ for all $y\in V(D')$.  By \Cref{lem:greedy}(a), there exists a subset $A\sub V(D')$ such that $D'[A]=D[A]$ is acyclic with $|A|\ge \frac{|V(D')|}{\Del+1}\ge |V(D)|^{1/3}$.   \Cref{lem:acyclic} then implies the existence of quasikernel with $|N^+[Q]|\ge |A|\ge |V(D)|^{1/3}$, proving the result.
	\end{proof}
\section{Proof of \Cref{thm:disjointWeak}: Disjoint $q$-kernels}\label{sec:disjoint}
	
	Recall that $S\sub V(D)$ is an $(r-1)$-source set if $N^-(S)=\emptyset$ and if $0<|S|\le r-1$.  In this subsection, we show that digraphs without $(r-1)$-source sets contain $r$ disjoint $q$-kernels with $q$ small.  We will in fact prove something slightly stronger.
	
	\begin{defn}
		Given a vector $\beta$ of length $r$ and a digraph $D$, we say that a tuple $(Q_1,\ldots,Q_r)$ of disjoint vertex subsets of $D$ is a \textit{$\beta$-kernel} of $D$ if $Q_i$ is a $\beta_i$-kernel of $D$ for all $i$.
	\end{defn}
	For example, a (3,2)-kernel of a digraph is a pair of disjoint sets $Q_1,Q_2$ such that $Q_1$ is a 3-kernel and $Q_2$ is a 2-kernel.
	
	\begin{thm}\label{thm:disjointFull}
		Define a sequence of vectors $\beta^{(r)}$ of length $r$ by taking $\beta^{(2)}=(3,2)$, and for $r>2$ by setting
		\[\beta_i^{(r)}=\begin{cases}
			2(\beta_{i}^{(r-1)}+r-2) & i<r,\\ 
			2r-3 & i=r.
		\end{cases}\]
		If $D$ contains no $(r-1)$-source sets for $r\ge 2$, then it contains a $\beta^{(r)}$-kernel.
	\end{thm}
	
	It is not difficult to determine the exact values of $\beta^{(r)}$ from \Cref{thm:disjointFull}.
	
	\begin{cor}\label{cor:disjointFull}
		If $D$ contains no $(r-1)$-source sets for $r\ge 2$, then it contains a $\beta^{(r)}$-kernel where
		\[\beta_i^{(r)}=\begin{cases}
			7\cdot 2^{r-2}-2r & i=1,\\ 
			3\cdot 2^{r-1}-2r & i=2,\\
			(4i-3) 2^{r-i}-2r & 3\le i\le r.
		\end{cases}\]
	\end{cor}
	Indeed, this follows by checking for each $i$ that the sequence $\beta_i^{(r)}$ in \Cref{thm:disjointFull} satisfies the same initial condition and recurrence relation as the sequence in \Cref{cor:disjointFull}.  From this our main result for disjoint $q$-kernels is easy to derive.
	
	\begin{proof}[Proof of \Cref{thm:disjointWeak} assuming \Cref{cor:disjointFull}]
		Recall that we wish to show that every digraph without $(r-1)$-sources contains $r$ disjoint $2^{r+1}$-kernels.  By \Cref{cor:disjointFull} we know there exist $r$ disjoint sets which are each $\beta_i^{(r)}$-kernels.  It is straightforward to show that $\beta_i^{(r)}\le 2^{r+1}$ for all $i$ (for example, by noting that $4i-3\le 2^{i+1}$ for  $i\ge 3$), giving the result.
	\end{proof}
	
	It remains to prove \Cref{thm:disjointFull}.  We begin by proving a slight strengthening of the  $r=2$ case which will serve as a warmup for the general approach.

	\begin{lem}\label{lem:2Disjoint}
		If $D$ is a source-free digraph, then for every independent set $Q\sub V(D)$, there exists a 3-kernel $Q'$ of $D$ which is disjoint from $Q$.
	\end{lem}
	Note that this together with the fact that every digraph contains a quasikernel implies that every source-free digraph contains a (3,2)-kernel.  This implies the $r=2$ case of \Cref{thm:disjointFull} as well as \Cref{prop:disjoint3}.
	\begin{proof}
		Define $D'=D-Q$ and let $Q'$ be any 2-kernel of $D'$.  Because $D$ is source-free and $Q$ is independent, every vertex of $Q$ has an in-neighbor in $V(D)\sm Q$. Since $Q'$ can reach every vertex of $V(D)\sm Q$ with a path of length at most 2 by construction, $Q'$ can reach every vertex of $V(D)$ with a path of length at most 3.  Since $Q'$ is independent by construction, we conclude that $Q'$ is a 3-kernel, giving the result.
	\end{proof}
	We would like to adapt the approach above to find 3 disjoint $q$-kernels in digraphs which contain no 2-source sets.  Specifically, we would like to prove that given any independent set $Q$, there always exist two disjoint $q$-kernels in $V(D)\sm Q$.
	
	Unfortunately this approach will not work.  Indeed, consider the digraph $D_s$ with vertex set $\{u,v_1,\ldots,v_s\}$ which consists of the arcs $uv_i$ and $v_i u$ for all $i$ (equivalently, $D_s$ is a star graph where each edge is replaced by a directed 2-cycle).  It is not difficult to check that  $D_s$ contains no 2-source sets if $s\ge 2$, but there exists an independent set $Q=\{v_1,\ldots,v_s\}$ such that $V(D_s)\sm Q=\{u\}$ does not contain two disjoint $q$-kernels.
	
	Roughly speaking, the  issue with the previous example is that, although $D_s$ contains no 2-source sets for $s\ge 2$, there exists an independent set $Q$ such that $D_s-Q$ contains sources, and hence there is no hope in finding two disjoint $q$-kernels in $D_s-Q$.  More generally, one can show that the following condition is an obstruction to the approach of \Cref{lem:2Disjoint}.

	\begin{defn}
		Given a digraph $D$, we say that a non-empty set of vertices $S$ is a \textit{pseduo-source set} if $N^-(S)$ is an independent set and if $S\cup N^-(S)$ is a source set (i.e.\ $N^-(S\cup N^-(S))=\emptyset$).  Equivalently, $S$ is a pseudo-source set if every $v\in N^-(S)$ has $N^-(v)\sub S$.  We say that such a set of vertices $S$ is an \textit{$s$-pseudo-source set} if $S$ is non-empty with $|S|\le s$.
	\end{defn}
	
	For example, every source set is a pseudo-source set (since in this case $N^-(S)=\emptyset$), and in particular digraphs without 1-pseudo-source sets are source-free.  More generally, one can check that a digraph $D$ contains a 1-pseudo-source set if and only if it contains an induced copy of $D_s$ for some $s\ge 0$ which has no incoming arcs from the rest of the digraph (noting in particular that the $s=0$ case is equivalent to $D$ containing a source).  
	
	It turns out that pseudo-source sets are the only obstruction to the approach of \Cref{lem:2Disjoint} going through.
	
	\begin{prop}\label{prop:noStructures}
		Let $\beta$ be a vector of length $r-1\ge 2$ such that every digraph $D'$ without $(r-2)$-source sets contains a $\beta$-kernel.  If $D$ is a digraph which contains no $(r-2)$-pseudo-source sets, then for every independent set $Q\sub V(D)$, there exist disjoint sets $Q_1,\ldots,Q_{r-1}$ which are disjoint from $Q$ such that each $Q_i$ is a $(2\beta_{i}+1)$-kernel for $D$.
	\end{prop}
	\begin{proof}
		Define a digraph $D'$ on $V(D)\sm Q$ by including the arc $uw$ if either $uw\in E(D)$ or if there exists a vertex $v\in Q$ such that $uv,vw\in E(D)$.
		\begin{claim}
			$D'$ contains no $(r-2)$-source sets.
		\end{claim}
		\begin{proof}
			Assume for contradiction that there exists a non-empty set $S\sub V(D')=V(D)\sm Q$ with $|S|\le r-2$ and $N^-_{D'}(S)=\emptyset$.  Our aim is to show that $S$ is an $(r-2)$-pseudo-source set in $D$, contradicting our hypothesis on $D$.  In particular, it suffices to show that each $v\in N^-_D(S)$ has $N^-_D(v)\sub S$.
			
			Consider an arbitrary $v\in N^-_{D}(S)$ (if such a vertex exists), say with $vw\in E(D)$ and $w\in S$.  If $v\in V(D)\sm Q$, then $vw\in E(D')$ by definition of $D'$ (since $w\in S\sub V(D)\sm Q$), contradicting $N^-_{D'}(S)=\emptyset$.  Thus we can assume $v\in Q$.  Note then that every $u\in N^-_D(v)$ satisfies $u\notin Q$ (since $v\in Q$ and $Q$ is independent) and $u\notin V(D')\sm S=V(D)\sm(S\cup Q)$ (since otherwise $v\in Q$ would imply $uw\in E(D')$ by definition of $D'$, contradicting $N_{D'}^-(S)=\emptyset$), so we must have $u\in S$.  Since $u\in N_D^-(v)$ was arbitrary we conclude $N_D^-(v)\sub S$, completing our contradiction for $S$ being an $(r-2)$-pseudo-source set.
		\end{proof}
		By hypothesis and the claim above, there exists a $\beta$-kernel $(Q_1,\ldots,Q_{r-1})$ of $D'$.  We claim that each $Q_i$ is a $(2\beta_{i}+1)$-kernel of $D$, which will complete the proof.  
		
		Indeed, fix some $i$ and first consider an arbitrary vertex $w\in V(D')$.  By definition of $Q_i$, there exists some $u\in Q_i$ and a directed walk from $u$ to $w$ of length at most $\beta_{i}$ in $D'$.  By definition of $D'$, each of these walks in $D'$ can be extended to a walk of length at most $2\beta_{i}$ in $D$. This gives the right distance condition  for vertices in $V(D')=V(D)\sm Q$, so it remains to show the result for $w\in Q$. 
		
		Because $D$ is source-free (since it contains no $(r-1)$-pseudo-source sets) and $Q$ is an independent set, every $w\in Q$ has some in-neighbor $v\in V(D)\sm Q=V(D')$.  By our previous analysis, we know $v$ can be reached from some $u\in Q_i$ by a walk of length at most $2\beta_{i}$, which means $w$ can be reached by a walk of length at most $2\beta_{i}+1$.  We conclude that $Q_i$ is a $(2\beta_{i}+1)$-kernel, proving the result.
	\end{proof}
	
	In order to use \Cref{prop:noStructures} towards proving \Cref{thm:disjointFull}, we must transform digraphs which avoid $(r-1)$-source sets into digraphs which avoid $(r-2)$-pseudo-source sets.  This is accomplished through the following lemma, where here we say that a pseudo-source set $S$ is \textit{proper} if there does not exist a non-empty set $T\subsetneq S$ which is a pseudo-source set. 
	
	\begin{lem}\label{lem:D'}
		Let $r\ge 3$ be an integer and $D$ a digraph without $(r-1)$-source sets.  Let $D'$ be the digraph obtained from $D$ by taking each proper $(r-2)$-pseudo-source set $S$ of $D$ and adding all possible arcs between the vertices of $N_D^-(S)$.  Then $D'$ contains no $(r-2)$-pseudo-source sets.
	\end{lem}
	
	\begin{proof}
		Assume for contradiction that there exists an $(r-2)$-pseudo-source set $S'$ of $D'$.  
		
		\begin{claim}
			$S'$ is an $(r-2)$-pseudo-source set of $D$.
		\end{claim}
		\begin{proof}
			If this is not the case, then by definition there exists $v\in N^-_D(S')$ with $N^-_D(v)\not\sub S'$.  This implies  $v\in N^-_{D'}(S')$ and $N^-_{D'}(v)\not \sub S'$ (because $D'$ only add arcs to $D$).  This implies  $S'$ is not an $(r-2)$-pseudo-source set of $D'$, a contradiction.
		\end{proof}
		\begin{claim}\label{cl:leastR}
			For any non-empty pseudo-source set $S$ of $D$, the set $S\cup N^-_D(S)$ contains at least $r$ vertices.
		\end{claim}
		\begin{proof}
			Observe that 
			\[N^-_D(S\cup N^-_D(S))\sub N^-_D(S)\cup N^-_D(N^-_D(S))\sub N^-_D(S)\cup S,\]
			where this last step uses that  $N^-_D(N^-_D(S))\sub S$ since $S$ is a pseudo-source set by hypothesis.  The expression above implies $S\cup N^-_D(S)$ is a source set of $D$, and since $D$ contains no $(r-1)$-source sets, this set must either be empty or have size at least $r$.  The set can not be empty since in particular $S\ne \emptyset$, so we conclude the set has size at least $r$ as desired.
			
		\end{proof}
		Let $S\sub S'$ be a proper $(r-2)$-pseudo-source set of $D$.
		\begin{claim}\label{cl:mostOne}
			There exists at most one vertex $w\in N_D^-(S)$ which is not in $S'$.
		\end{claim}
		\begin{proof}
			If there exists no vertex with this property then we are done, so we may  assume there exists a vertex $w$ with this property.  Note that $w\in N_{D'}^-(S')$ since, by hypothesis, $w\notin S'$ and $w$ has an arc in $D\sub D'$ to some vertex in $S\sub S'$.  Thus 
			\[(N^-_D(S)\sm \{w\})\sub N^-_{D'}(w)\sub S',\]
			where the first inclusion uses the definition of $D'$, and the second uses that $S'$ is a pseudo-source set of $D'$ and $w\in N_{D'}^-(S)$.  This proves the claim.
		\end{proof}
		
		Putting these last two claims together, we see that $S'$ contains at least $r-1$ vertices (namely $S\cup N^-_D(S)\sm \{w\}\sub S'$ since $S\sub S'$ by definition), contradicting $S'$ being an $(r-2)$-pseudo-source set.  We conclude the result.
		
	\end{proof}
	
	To use \Cref{lem:D'}, we will need to understand how much the diameters of $D$ and $D'$ differ from each other.  For this, we say a set of vertices $C$ of a digraph $D$ is \textit{strongly connected} if for all $u,v\in C$ there exists a directed path from $u$ to $v$, and we say $C$ is a \textit{strongly connected component} if $C$ is a maximal set with this property.  We define the \textit{diameter} of a strongly connected set $C$ to be $\max_{u,v\in C} \dist_D(u,v)$, and we note that for strongly connected components this quantity is the same as $\max_{u,v\in C} \dist_{D[C]}(u,v)$ (since no vertex $w$ outside of $C$ can ever be used in a shortest path from $u$ to $v$, as otherwise $C\cup \{w\}$ would also be strongly connected).
	
	\begin{lem}\label{lem:walk}
		If $S$ is a proper pseudo-source set of a source-free digraph $D$, then $S\cup N^-(S)$ is a strongly connected component of $D$ and has diameter at most $2|S|$.
	\end{lem}

	This diameter bound is best possible: if $D$ is obtained from a path graph $u_0u_1\cdots u_{2r-1}u_{2r}$ by including both arcs for each edge, then $S=\{u_1,u_3,\ldots, u_{2r-1}\}$ is a proper pseudo-source set and $S\cup N^-(S)=V(D)$ has diameter exactly $2r$.  
	
	\begin{proof}
		We begin with the following.
		\begin{claim}\label{cl:scc}
			For any two vertices $u,v\in S\cup N^-(S)$, there exists a directed path in $D$ from $u$ to $v$.
		\end{claim}
		\begin{proof}
			Let $v\in S\cup N^-(S)$ be arbitrary, and let $T\sub S\cup N^-(S)$ denote the set of vertices $u$ such that there exists a directed path from $u$ to $v$ in $D$, noting that we trivially include $v\in T$.  We aim to show that $S\cap T$ is a pseudo-source set. 
			
			Showing $S\cap T\ne \emptyset$ is trivial if $v\in S$.   If instead $v\in N^-(S)$, then $N^-(v)\sub S$ since $S$ is a pseudo-source set.  We also have $N^-(v)\sub T$ by definition of $T$, so $N^-(v)\sub S\cap T$.  Moreover, $N^-(v)\ne \emptyset$ since $D$ is source-free, proving that $S\cap T\ne \emptyset$.
			
			We next show that every $u\in N^-(S\cap T)$ has $N^-(u)\sub S\cap T$.   If $u\in N^-(S\cap T)$ then, since $u$ has an arc to a vertex of $T$, we must have $u\in T$ by definition of $T$, and similarly we must have $N^-(u)\sub T$.  Because $u\in T$ and $u\notin S\cap T$ by definition of $u\in N^-(S\cap T)$, we must have $u\notin S$, and as such $u$ having an arc to $S$ implies that $u\in N^-(S)$, so $N^-(u)\sub S$ by definition of $S$ being a pseudo-source set.  In total we find $N^-(u)\sub S\cap T$ for all $u\in N^-(S\cap T)$, which combined with our observation $S\cap T\ne \emptyset$  shows that $S\cap T$ is a pseudo-source set.
			
			Because $S\cap T$ is a  pseudo-source set and $S$ is proper, we must have $S\cap T=S$, i.e.\ $S\sub T$.  By definition of $T$ this implies $S\cup N^-(S)\sub T$, and hence $T=S\cup N^-(S)$.  This implies every vertex $u$ of $S\cup N^-(S)$ can reach the (arbitrary) vertex $v\in S\cup N^-(S)$, proving the claim.
		\end{proof}
		Observe that no vertex $u\notin S\cup N^-(S)$ has an arc to $S\cup N^-(S)$ by definition of $S$ being a pseudo-source set.  This together with the claim above implies $S\cup N^-(S)$ is a strongly connected component.

		To prove the diameter condition, consider any path  $x_0x_1\cdots x_\ell$ in $D[S\cup N^-(S)]$.  Note that for all $i$, at least one of the vertices $x_i,x_{i+1}$ must lie in $S$ since $N^-(v)\sub S$ for all $v\in N^-(S)$, and this implies $\ell\le 2|S|$.
	\end{proof}

	This gives the following.
	\begin{lem}\label{cor:paths}
		Let $r,\ D$, and $D'$ be as in \Cref{lem:D'}.  If $x_0\cdots x_\ell$ is a directed path in $D'$, then there exists a directed path from $x_0$ to $x_\ell$ in $D$ of length at most $\ell+2(r-2)-1$.
	\end{lem}
	\begin{proof}
		This proof relies on two claims.
		\begin{claim}
			If $S$ is a proper pseudo-source set of $D$, then $N_{D'}^-(S\cup N^-_D(S))=\emptyset$.
		\end{claim}
		\begin{proof}
			Note that $N_D^-(S\cup N^-_D(S))=\emptyset$ by definition of $S$ being a pseudo-source set of $D$.  Thus the only way the claim can fail is if there exists some $u\notin S\cup N^-_D(S)$ for which $D'$ contains an arc from $u$ to $S\cup N^-_D(S)$.  By definition of $D'$, this would only be possible if there exists some proper pseudo-source set $T$ of $D$ such that $T\cup N^-_D(T)$ contains $u$ and intersects $S\cup N^-_D(S)$.  However, the two sets $T\cup N^-_D(T)$ and $S\cup N^-_D(S)$ are strongly connected components of $D$ by \Cref{lem:walk}, so these two sets must either be equal or disjoint from each other, a contradiction to $T\cup N^-_D(T)$ intersecting $S\cup N^-_D(S)$ and containing a vertex $u\notin S\cup N^-_D(S)$. 
		\end{proof}
		
		The key observation from this claim is that if $j$ is such that $x_j\notin S\cup N^-_D(S)$ for some proper $(r-2)$-pseudo-source set $S$ of  $D$, then $x_{j+1}\notin S\cup N^-_D(S)$, since the claim says that the only arcs in $D'$ to vertices of $S\cup N^-_D(S)$ come from vertices of $S\cup N^-_D(S)$.  With this key observation we get the following. 
		\begin{claim}
			There exists an integer $i$ and a proper $(r-2)$-pseudo-source set $S$ of $D$ such that for all $j\ge i$ we have $x_{j-1}x_j\in E(D)$, and for all $j<i$ we have $x_j\in S\cup N^-_D(S)$.
		\end{claim}
		\begin{proof}
			By the key observation mentioned above, there exists some smallest integer $i\ge 0$ (possibly with $i=\ell+1$) such that for all $j\ge i$, $x_j\notin S\cup N^-_D(S)$ for every proper $(r-2)$-pseudo-source set $S$ of $D$.  In particular, this implies $x_{j-1}x_j\in E(D)$ for all $j\ge i$, since the only arcs which lie in $E(D')\sm E(D)$ contain two vertices in some $S\cup N^-_D(S)$ set with $S$ a proper $(r-2)$-pseudo-source set of $D$.
			
			If $i=0$ then the $j<i$ part of the claim trivially holds, so assume $i\ge 1$.  By definition of $i$, there exists some proper $(r-2)$-pseudo-source set $S$ of $D$ with $x_{i-1}\in S\cup N^-_D(S)$.  Using the key observation together with induction shows that we must have $x_j\in S\cup N^-_D(S)$ for all $j<i$, completing the claim.
		\end{proof}
		Let $i,S$ be as in the claim above.  If $i\le 1$ then $x_0\cdots x_\ell$ itself is a path in $D$ and there is nothing to prove, so assume $i\ge 2$.  By \Cref{lem:walk} (which we can apply since $D$ avoiding $(r-1)$-source sets means it is source-free), there exists a walk of length at most $2(r-2)$ from $x_0$ to $x_{i-1}$ in $D$.  Composing this with the walk $x_{i-1}\cdots x_\ell$ (which is also in $D$ by definition of $i$) gives a path of length at most $2(r-2)+\ell-(i-1)\le 2(r-2)+\ell-1$ in $D$, giving the result.
	\end{proof}

	With everything in place, we can now prove our main result
	\begin{proof}
		Recall that we wish to prove that if $D$ contains no $(r-1)$-source sets for $r\ge 2$, then it contains a $\beta^{(r)}$-kernel where $\beta^{(2)}=(3,2)$ and inductively  \[\beta_i^{(r)}= 2(\beta_{i}^{(r-1)}+r-2)\ \ \tr{for }i<r,\hspace{2em} \beta_r^{(r)}=2r-3.\]
		We prove this by induction on $r$, the $r=2$ case following from \Cref{lem:2Disjoint} and the fact that every digraph contains a 2-kernel.
		
		Assume we have proven the result up to some value $r\ge 3$.  Let $D$ be a digraph without $(r-1)$-source sets and let $D'$ be as in \Cref{lem:D'}, i.e.\ the digraph obtained by taking each proper $(r-2)$-pseudo-source set $S$ of $D$ and adding all possible arcs between the vertices of $N^-_D(S)$, and let $Q_r$ be an arbitrary 2-kernel of $D'$. 
		Because $D'$ has no $(r-2)$-pseudo-source sets by \Cref{lem:D'}, \Cref{prop:noStructures} implies that there exist a $\gam$-kernel $(Q_1,\ldots,Q_{r-1})$ for $D'$ where $\gam_i:=2\beta_i^{(r-1)}+1$ for all $i$ such that each $Q_i$ is disjoint from $Q_r$
		
		We claim that each $Q_i$ is a $2(\beta^{(r-1)}_i+r-2)$-kernel of $D$ for all $i<r$.  Indeed, since $D\sub D'$, each of these sets are still independent.  Fix an integer $i$ and $v\in V(D)$ and consider a shortest walk $x_0\cdots x_\ell$ in $D'$ from a vertex $u\in Q_i$ to $v$, noting that $\ell\le 2\beta_i^{(r-1)}+1$ by hypothesis.  By \Cref{cor:paths}, there exists a path from $x_0$ to $x_\ell$ in $D$ of length at most $ 2\beta_i^{(r-1)}+1+2(r-2)-1=2(\beta_i^{(r-1)}+r-2)$, proving the claim.
		
		An essentially identical argument shows that $Q_r$ is a $q$-kernel of $D$ with $q=2+2(r-2)-1=2r-3$, which combined with the claim above gives the desired result.
	\end{proof}
	
	\section{Proof of \Cref{thm:bipartite}: Bipartite Digraphs}\label{sec:bipartite}
	
	In this section we show that bipartite digraphs with high girth have small $q$-kernels.  We begin by proving  our result for quasikernels,  for which the following will be useful.
	\begin{lem}\label{lem:disjointUnion}
		If $D$ is a digraph such that for every vertex $u\in V(D)$ there exists a vertex $v\in N^+(u)$ which has in-degree 1, then $D$ is a disjoint union of directed cycles.
	\end{lem}
	\begin{proof}
		Since by hypothesis every vertex in $V(D)$ has an out-neighbor with in-degree 1, and since by definition each vertex of in-degree 1 has only 1 in-neighbor, it follows that there are at least (and hence exactly) $|V(D)|$ vertices with in-degree 1.  It follows that every vertex also has out-degree 1 (using that every vertex has positive out-degree by hypothesis and that the sum of the in-degrees equals the sum of the out-degrees), from which we conclude that $D$ is a disjoint union of directed cycles.
	\end{proof}
	With this we can prove our result for small quasikernels.
	\begin{proof}[Proof of \Cref{prop:smallQuasiBipartite}]
		Recall that we wish to show that if $D$ is a source-free digraph with bipartition $U\cup V$ which is not the disjoint union of directed 2-cycles and 4-cycles, then there exists a quasikernel $Q$ such that $|Q|<\half |V(D)|$.  If  $|U|<\half |V(D)|$, then taking $Q=U$ is a quasikernel (since $D$ is source-free and bipartite) which satisfies the result.  A similar result holds if $|V|<\half |V(D)|$, so we may assume $|U|=|V|=\half |V(D)|$.
		
		First consider the case that $D$ is a directed cycle on $x_1,\ldots,x_{2\ell}$ for some $\ell\ge 3$.  In this case the set $Q=\{x_1\}\cup \{x_{2i}:2\le i\le \ell-1\}$ is a quasikernel with $|Q|=\ell-1<\half |V(D)|$.  A similar argument works if $D$ is a disjoint union of directed cycles (which necessarily has at least one cycle of length at least 6 by hypothesis), so we may assume $D$ is not of this form.  The previous lemma thus implies that there exists a vertex $u\in V(D)$ such that every vertex in $N^+(u)$ has in-degree at least 2, and without loss of generality we may assume $u\in U$.
		
		Consider $Q=U\sm \{u\}$, and observe that $Q$ is a quasikernel due to both our hypothesis on $u$ (which ensures that every vertex of $V$ is within distance 1 of $Q$) and our hypothesis that $D$ is source-free (which ensures $u$ is within distance 1 of some vertex in $V$).  We also have $|Q|=\half |V(D)|-1$ by hypothesis on $U$, proving the result.
	\end{proof}

	We now prove our results about $q$-kernels.  Here the key idea will be to reduce the problem to the following case.
	
	\begin{defn}
		We say that a finite digraph $D$ is \textit{unicyclic} if every vertex has in-degree 1 and if the underlying graph of $D$ is connected.
	\end{defn}
	The name ``unicyclic'' is motivated by the following result.
	\begin{lem}\label{lem:unicyclic}
		If $D$ is unicyclic, then (1) $D$ contains exactly one directed cycle $C$, and (2) for every vertex $w\in V(D)$, there exists a unique directed path $P_w$ whose first vertex is in $C$, last vertex is $w$, and which contains no other vertices of $C$.
	\end{lem}
	\begin{proof}
		Let $w_1\in V(D)$ be arbitrary and iteratively define $w_{i+1}$ to be the unique vertex in $N^-(v_i)$.  Because $D$ is finite, we must have $w_i=w_j$ for some $i<j$.  If one chooses this $i,j$ so that $j-i$ is as small as possible, then we see that $w_j,w_{j-1},\ldots,w_{i+1}$ are distinct vertices with $w_{k+1}w_k\in E(D)$ for all $i<k<j$ and with $w_{i+1}w_j\in E(D)$.  We conclude that $D$ contains at least one directed cycle.  Moreover, the directed path $w_i,w_{i-1}\ldots,w_1$ shows that every vertex of $D$ has a directed path to it from a directed cycle which intersects this cycle only at its endpoint.  It thus remains to prove uniqueness of these cycles and paths.

		Let $C=(v_j,v_{j-1},\ldots v_1)$ be a directed cycle and assume for contradiction that there exists some other directed cycle $C'$ which contains $v_1$.  Because $v_1$ has a unique in-neighbor, namely $v_2$, we see that the in-neighbor of $v_1$ in $C'$ must also be $v_2$, and continuing this we see that $C'=C$, a contradiction.  Thus if another directed cycle $C'$ exists it must be disjoint from $C$.  Because $D$ is unicyclic, there exists some shortest (undirected) path $u_1\cdots u_t$ in the underlying graph of $D$ with $u_1\in C$ and $u_t\in C'$.  Because this is a shortest path, we must have $u_2\notin C$, and because the unique in-neighbor of $u_1$ lies in $C$ this implies $u_1u_2\in E(D)$.  Continuing this argument by induction, we conclude that $u_iu_{i+1}\in E(D)$ for all $i$, but this implies $u_t$ has an in-neighbor that is not in $C'$, a contradiction to the in-degree 1 hypothesis.  We conclude that $C$ is the only directed cycle in $D$.
		
		Finally, let $w\in V(D)$ and suppose there exists two directed paths $u_su_{s-1}\cdots u_1$ and $v_tv_{t-1}\cdots v_1$ with the properties as in the lemma, noting that this means $u_1=v_1=w$.  Iterativelly given that $u_i=v_i$ for some $i<\min\{s,t\}$, we must have that $u_{i+1}=v_{i+1}$ since $u_i=v_i$ has a unique in-neighbor.  If, say, $s<t$, then this implies the path $v_t\cdots v_1$ contains two vertices of $C$, namely  $v_s=u_s$ and $v_{s+1}$ (since this is the unique in-neighbor of $v_s=u_s$), a contradiction.  Thus we must have $s=t$ and that $u_i=v_i$ for all $i$, proving  these directed paths are unique.
	\end{proof}
	We now prove the following special case of \Cref{thm:bipartite} when $D$ is unicyclic and $q$ is odd, which will turn out to give the full result after a few reductions.

	\begin{prop}\label{prop:unicyclic}
		Let $D$ be a unicyclic digraph with directed cycle of length $2\ell\ge 4$ and which has bipartition $U\cup V$.  If $q\ge 3$ is an odd integer and $|V(D)|\ge (q+3)/2$, then there exists a $q$-kernel $Q\sub U$ with $|Q|\le \frac{2}{q+3}|V(D)|$.
	\end{prop}
	We note that this statement differs slightly from \Cref{thm:bipartite} in that we have replaced the conditions on cycle lengths of $D$ with the condition $|V(D)|\ge (q+3)/2$ (which will allow our inductive  proof to go through  smoother), and we have changed what the notation $\ell$ represent.  We also note that this result is false for $\ell=1$, as $D$ could consist of a directed 2-cycle $uv$ together with a directed path of $q$ vertices attached to $v$.
	\begin{proof}
		Assume for contradiction that there exists a counterexample $D$, and choose such a counterexample with $|V(D)|$ as small as possible.  Let $C$ denote the vertices of the directed cycle of $D$, and for each $w\in V(D)$, let $P_w$ denote the path to $w$ guaranteed by \Cref{lem:unicyclic}.  We begin by showing that every vertex is somewhat close to the cycle.

		\begin{claim}\label{cl:weakDistance}
			We have $\dist(C,w)<q$ for all $w\in V(D)$.
		\end{claim}
		For this proof, we will make multiple uses of the fact that if $w'\in P_w$, then
		\[\dist(C,w)=\dist(C,w')+\dist(w',w),\]
		since $P_w$ being the unique path from \Cref{lem:unicyclic} implies that $P_w$ must be composed of $P_{w'}$ together with the unique path from $w'$ to $w$ if $w'\notin C$ (and if $w'\in C\cap P_w$ then trivially $\dist(C,w)=\dist(w',w)$).
		
		\begin{proof}
			Let $w\in 
			V(D)$ be such that $\dist(C,w)=\max_x \dist(C,x)$ and assume for contradiction that $\dist(C,w)\ge q$.  Since $P_w$ has length $\dist(C,w)\ge q\ge  (q+3)/2$ (using $q\ge 3$) and $q$ is odd, there exists a vertex $w'\in P_w$ with $\dist(w',w)=(q+1)/2$, noting that \[\dist(C,w')=\dist(C,w)-\dist(w',w)>0,\] so $w'\notin C$.   Let $X=\{x\in V(D):\dist(w',x)\le (q+1)/2\}$ and define $D'=D-X$.  We claim that $D'$ satisfies the conditions of the lemma.
			
			To show $D'$ is unicyclic, it suffices to show that each path $P_y$ with $y\in V(D')$ is contained in $V(D')$ (since this will imply each $y$ continues to have in-degree 1 and that the underlying graph is connected).  Assume for contradiction we had some $y\notin X$ and  $x\in P_y\cap X$.  Because $x\in X$, there exists a path from $w'$ to $x$, and having $x\in P_y$ implies there is also a path from $w'$ to $y$ by transitivity.  Because $y\notin X$, we must have $\dist(w',y)>(q+1)/2=\dist(w',w)$.  But this implies
			\[\dist(C,y)=\dist(C,w')+\dist(w',y)>\dist(C,w')+\dist(w',w)=\dist(C,w),\]
			a contradiction to how we chose $w$.
			
			We next show $|V(D')|\ge (q+3)/2$. For this, we	observe that the number of vertices removed from $P_w$ is exactly $\dist(w',w)+1=(q+3)/2$ and that we removed none of the $2\ell$ vertices in the cycle, so we have
			\[|V(D')|\ge 2\ell+\dist(C,w)-(q+3)/2\ge 4+q-(q+3)/2=(q+5)/2,\]
			proving the desired bound.  Finally, since $D$ has bipartition $U\cup V$, it is clear that $D'$ has bipartition $U'\cup V'$ with e.g.\ $U'=U\cap V(D')$.
			
			We thus see that $D'$ satisfies the conditions of the proposition.  Because $D$ was a minimal counterexample,  there must exist a $q$-kernel $Q'\sub U'$ of $D'$ satisfying \[|Q'|\le \frac{2}{q+3}|V(D')|=\frac{2}{q+3}(|V(D)|-|X|)\le \frac{2}{q+3}|V(D)|-1,\]
			with this last step using $|X|\ge (q+3)/2$ (since $X$ includes every vertex along the path from $w'$ to $w$).  
			
			Define $z$ to be $w'$ if $w'\in U$ and $z$ to be the unique vertex in $N^-_D(w')$ otherwise.  Note that in either case $z\in U$.  Define $Q=Q'\cup \{z\}\sub U$, noting that this satisfies $|Q|\le \frac{2}{q+3}|V(D)|$ by the inequality above.  We also have for $x\notin X$ that 
			\[\dist_D(Q,x)\le \dist_{D'}(Q',x)\le q,\]
			and for $x\in X$ we have
			\[\dist(Q,x)\le \dist(z,x)\le \dist(z,w')+\dist(w',x)\le 1+(q+1)/2\le q.\]
			Thus $Q$ is a $q$-kernel satisfying the conditions of the proposition, contradicting $D$ being a counterexample.  We  conclude the claim.	
		\end{proof}
		We can boost the previous claim by showing there are relatively few vertices outside the cycle.  For this we introduce some new notation that will be vital for the rest of the proof. 
		
		Label the vertices of the cycle $C$ by $(u_1,v_1,\ldots,u_\ell,v_\ell)$ with $u_j\in U$ for all $j$.  We say that a vertex $w\in V(D)$ has \textit{type $j$} if the path $P_w$ begins with either $u_j$ or $v_j$.  Equivalently, $w$ is of type $j$ if $\dist(C,w)=\dist(\{u_j,v_j\},w)$.  We let $\ty(j)$ denote the number of vertices of type $j$ in $D$.  A useful observation that we  use throughout the proof is that if $w$ is of type $j$, then
		\begin{equation}\dist(u_j,w)\le \ty(j)-1,\label{eq:chunkDistance}\end{equation}
		as the shortest path from $u_j$ to $w$ only uses vertices of type $j$ (and the length of this path will be 1 less than the total number of vertices in the path).  With this observation in mind, we see that the following significantly improves upon the previous claim.
		
		\begin{claim}\label{cl:mediumDistance}
			We have $\ty(j)<(q+3)/2$ for all $j$.
		\end{claim}
		\begin{proof}
			Assume for contradiction that there exists some $j$ with $\ty(j)\ge (q+3)/2$.  
			
			First consider the case $|V(D)|-\ty(j)\le (q+1)/2$, and let $Q=\{u_j\}\sub U$.  Observe that $|Q|\le \frac{2}{q+3}|V(D)|$ since we assume $|V(D)|\ge (q+3)/2$.  Also observe that if $w\in V(D)$ is of type $j'\ne j$, then there exists a shortest path from $u_j$ to $w$ obtained by using the arc $u_jv_j$ followed by some set of vertices not of type $j$, so
			\[\dist(u_j,w)\le 1+(|V(D)|-\ty(j))\le 1+(q+1)/2\le q.\]
			On the other hand, if $w$ is of type $j$, then \Cref{cl:weakDistance} implies
			\[\dist(u_j,w)\le 1+(q-1)\le q,\]
			where the $1$ comes from the fact that $P_w$ may begin with $v_j$.
			We conclude that $Q\sub U$ is a $q$-kernel for $D$ of size at most $\frac{2}{q+3}|V(D)|$, a contradiction to $D$ being a counterexample to the proposition.  Thus we may assume $|V(D)|-\ty(j)\ge (q+3)/2$.
			
			Next consider the subcase $\ell=2$ and define $Q=\{u_1,u_2\}\sub U$.  Observe that our hypothesis $\ty(j),|V(D)|-\ty(j)\ge (q+3)/2$ imply $V(D)|\ge q+3$, so $|Q|\le \frac{2}{q+3}|V(D)|$.  Moreover, if $w\in V(D)$ has type $j$, then again by using \Cref{cl:weakDistance} we find
			\[\dist(Q,w)\le \dist(u_j,w)\le q,\]
			so $Q$ is a $q$-kernel, contradicting $D$ being a counterexample.
			
			With this we may assume $\ell>2$ and $|V(D)|-\ty(j)\ge (q+3)/2$.  Let $D'$ be the digraph obtained from $D$ by deleting all vertices of type $j$ and then adding the arc $v_{j-1}u_{j+1}$ (with these indices written mod $\ell$).  It is not difficult to see that $D'$ is a unicyclic digraph with cycle length $2(\ell-1)\ge 4$ which has $|V(D')|=|V(D)|-\ty(j)\ge (q+3)/2$ and which has bipartition $U'\cup V'$ defined by e.g.\ $U'=U\cap V(D')$.  
			
			Since $D'$ satisfies the conditions of the proposition, and since we choose $D$ to be a minimal counterexample, there must exist some $q$-kernel $Q'\sub U'$ for $D'$ of size at most 
			\[\frac{2}{q+3}|V(D')|=\frac{2}{q+3}(|V(D)|-\ty(j))\le \frac{2}{q+3}|V(D)|-1.\]
			Let $Q:=Q'\cup \{u_j\}\sub U$, noting that $|Q|\le \frac{2}{q+3}|V(D)|$ by the above.  We claim that $Q$ is a $q$-kernel for $D$.  
			
			Indeed, if $w\in V(D)$ is of type $j$ then \Cref{cl:weakDistance} implies $\dist(u_j,w)\le q$.  Otherwise $w\in V(D')$, so there exists a directed path $P'=(w_1,\ldots,w_t)$ from $Q'$ to $w$ in $D'$ of length at most $q$.  If the arc $v_{j-1}u_{j+1}$ is not in $P'$, then every arc of $P'$ is still in $D$, so $P'\sub D$ and the distance condition is satisfied. Thus we may assume $v_{j-1}u_{j+1}$ is in $P'$, say with $w_s=v_{j-1}$, noting that $s\ge 2$ since any path from $Q'\sub U$ must start with a vertex in $U$.  In this case we consider the path $P=(u_j,v_j,w_{s+1},\ldots,w_t)$ (i.e. we delete the beginning of $P'$ up to $u_{j+1}$ and then append $u_j,v_j$ at the start).  Note that \[\mathrm{len}(P)=\mathrm{len}(P')+2-s\le \mathrm{len}(P')\le q,\] so this path shows $\dist_D(Q,w)\le q$.
			
			In total we conclude that $Q\sub U$ is a small $q$-kernel for $D$, a contradiction, giving the claim.
		\end{proof}
		Define $m=\lceil(q+3)/4\rceil$.  Equivalently, $m\ge 2$ is the integer such that either $q=-3+4m$ or $q=-5+4m$.  The idea here is that if we can find some $Q\sub U$ which uses a $1/m$ fraction of the $\ell$ vertices $u_j$, then it turns out $Q$ will have the desired size.  This idea quickly gives the following.
		\begin{claim}\label{cl:smallD}
			We have $\frac{2}{q+3}|V(D)|<\lceil \ell/m\rceil$.
		\end{claim}
		\begin{proof}
			Assume for contradiction that $\frac{2}{q+3}|V(D)|\ge \lceil \ell/m\rceil$.  Define \[Q=\{u_{1+mt}:0\le t\le \lceil \ell/m\rceil-1\}\sub U,\] noting that $|Q|=\lceil \ell/m\rceil\le\frac{2}{q+3}|V(D)|$.  
			
			We aim to show $Q$ is a $q$-kernel.  To this end, consider any $w\in V(D)$, say of type $j$.  Observe that there exists some integer $0\le t<\lceil \ell/m\rceil$ such that $1+mt\le j\le m(t+1)$ (namely $t= \lfloor (j-1)/m \rfloor$ works via observing $j-m\le mt\le j-1$).  This and \eqref{eq:chunkDistance} implies
			\[\dist(Q,w)\le \dist(u_{1+mt},u_j)+\dist(u_j,w)\le 2(j-(1+mt))+\ty(j)-1\le 2(m-1)+(q-1)/2,\]
			with this last step using $j\le m(t+1)$ and \Cref{cl:mediumDistance}.  Note that this quantity is at most $q$ since $q\ge -5+4m$.  Thus $Q\sub U$ is a small $q$-kernel for $D$, a contradiction.
		\end{proof}
		
		We will use the bound above to prove the following set of inequalities, where here and throughout we let $0\le a\le m-1$ be the integer such that $\ell\equiv a\mod m$.
		\begin{claim}\label{cl:typeA}
			For all $j$ we have
			\[\ty(j)\le q-2m-2a+5,\]
			and further
			\[\min_j \ty(j)\le q-2m-2a+3.\]
		\end{claim}
		\begin{proof}
			We begin by making a few observations about $|V(D)|$.  By \Cref{cl:smallD}, we have \[\frac{2}{q+3}|V(D)|<\lceil \ell/m\rceil\le \lfloor \ell/m\rfloor+1=(\ell+m-a)/m.\]
			Recall that $|V(D)|\ge (q+3)/2$ by hypothesis of the proposition, so the inequality above combined with $\ell\equiv a\mod m$ implies \[\ell \ge m+a.\]  Using these two inequalities gives
			\[|V(D)|<\frac{(\ell+m-a)(q+3)}{2m}=(2\ell+2m-2a)+\frac{(\ell+m-a)(q+3-4m)}{2m}\le (2\ell+2m-2a)+(q+3-4m),\]
			where this last step used $q+3-4m\le 0$ and $\ell\ge m+a$.  In total, this implies
			\begin{equation}|V(D)|\le 2\ell+q-2m-2a+2.\label{eq:smallD}\end{equation}
			We also note that 
			\[q-2m-2a+2>0.\]
			Indeed, recalling that either $q=-5+4m$ or $q=-3+4m$, we see that this immediately holds if either $a\le m-2$ or if $a=m-1$ and $q= -3+4m$, so the only potential issue is when $a=m-1$ and $q=-5+4m$.  In this case, \eqref{eq:smallD} would imply $|V(D)|\le 2\ell-1$, a contradiction to having $|V(D)|\ge 2\ell$ by virtue of $D$ containing a directed $2\ell$-cycle.
			
			Returning to the proof of the claim; if there existed  some $j$ with $\ty(j)\ge q-2m-2a+6$, then this together with the trivial bound $\ty(j')\ge 2$ for all $j'$ (coming from $u_j,v_j$) would imply
			\[|V(D)|=\sum \ty(j')\ge 2(\ell-1)+q-2m-2a+6=2\ell+q-2m-2a+4,\]
			a contradiction to \eqref{eq:smallD}.  This proves the first part of the claim.
			
			Similarly, if we had $\ty(j)\ge q-2m-2a+4$ for all $j$, then this would imply
			\begin{align*}|V(D)|&\ge \ell(q-2m-2a+4)=2\ell+\ell(q-2m-2a+2)\\ &=2\ell+q-2m-2a+2+(\ell-1)(q-2m-2a+2)>2\ell+q-2m-2a+2,\end{align*}
			where this last step used $q-2m-2a+2>0$ and $\ell\ge  2$.  This contradicts \eqref{eq:smallD}, so we conclude the result.
		\end{proof}
		
		With these inequalities and the definition of $a$ in mind, we assume without loss of generality that $\ty(\ell)=\min_j \ty(j)$ and define  \[Q=\{u_{1+mt}:0\le t\le \lfloor\ell/m\rfloor-1\}=\{u_{1+mt}:0\le t\le (\ell-a)/m-1\}\sub U.\]  Note that \[|Q|=\lfloor \ell/m\rfloor\le \frac{1}{2m}\cdot 2\ell\le  \frac{2}{q+3}\cdot 2\ell \le \frac{2}{q+3}|V(D)|,\]
		where this first inequality used that $\frac{2}{q+3}$ either equals $\frac{1}{2m}$ if $q=-3+4m$ or $\frac{1}{2m-1}\ge \frac{1}{2m}$ if $q=-5+4m$.  Note that this inequality together with \Cref{cl:smallD} implies $a\ne 0$ (i.e. that $\lfloor \ell/m\rfloor\ne \lceil \ell/m\rceil$).
		
		We aim to show that $Q$ is a $q$-kernel.  First, we claim that if $w\in V(D)$ is of type $j<\ell$, then there exists some $0\le t\le (\ell-a)/m-1$ such that $1+mt\le j\le a-1+m(t+1)$.  Indeed, if $j\le \ell-a$ then one can take $t= \lfloor (j-1)/m \rfloor\le (\ell-a)/m-1$, and as in the proof of \Cref{cl:smallD} this satisfies  \[1+mt\le j\le m(t+1)\le a-1+m(t+1).\]  If instead $\ell-a<j<\ell$ then we take $t=(\ell-a)/m-1$, in which case \[a-1+m(t+1)=\ell-1\ge j\ge 1+mt.\]  Using this together with \eqref{eq:chunkDistance}  and \Cref{cl:typeA}  gives
		\begin{align*}\dist(Q,w)&\le \dist(u_{1+mt},u_j)+\dist(u_j,w)\\ &\le 2(j-(1+mt))+\ty(j)-1\\ &\le 2(m+a-2)+(q-2m-2a+5)-1= q,\end{align*}
		where this last inequality used $j\le a-1+m(t+1)$.  Similarly if $w$ has type $\ell$ then we have
		\begin{align*}\dist(Q,w)&\le \dist(u_{1+\ell-a-m},u_\ell)+\dist(u_\ell,w)\\ &\le 2(\ell-(1+\ell-a-m))+\ty(\ell)-1\\ &\le 2(m+a-1)+(q-2m-2a+3)-1 =q.\end{align*}
		We conclude that $Q$ is a small $q$-kernel for $D$, a contradiction.  We conclude that no counterexample $D$ exists, proving the result.
	\end{proof}
	With \Cref{prop:unicyclic} in hand we can prove \Cref{thm:bipartite} in the special case when $q$ is odd and $\ell=(q+3)/2$.
	
	\begin{thm}\label{thm:bipartiteSpecial}
		Let $D$ be a source-free digraph with kernel $K$.  If $q\ge 3$ is odd and every directed cycle of $D$ of even length has length at least $(q+3)/2$, then $D$ contains a $q$-kernel $Q$ of size at most $\frac{2}{q+3}|V(D)|$ with $Q\sub K$.
	\end{thm}
	\begin{proof}
		Let $D,q$ be as in the hypothesis of the theorem.  Observe that if we can find some $Q\sub K$ which is a $q$-kernel for some spanning subgraph $D'\sub D$, then $Q$ will still be a $q$-kernel of $D$ since $K$ is independent in $D$ (and so will $Q$).  With this in mind, we first define $\tilde{D}'\sub D$ by deleting every arc which does not use a vertex of $K$, noting that this digraph continues to be source-free since $K$ is a kernel (which means every vertex outside of $K$ still has an in-neighbor from $K$, and every vertex inside $K$ still has its previous incoming arcs).  We then define a subdigraph $D'\sub \tilde{D}'$ by iteratively deleting arcs from $\tilde{D}'$ until every vertex has in-degree exactly 1 (and this can be done since $\tilde{D}'$ is source-free).

		We can write $D'$ as the disjoint union of digraphs $D_1\cup \cdots \cup D_t$ where the underlying graph of each $D_i$ is connected, and we observe that each $D_i$ is unicyclic and bipartite with bipartition $U_i\cup V_i$ defined by $U_i=V(D_i)\cap K$ and $V_i=V(D_i)\sm K$.  Moreover, since every directed cycle of $D$ of even length
		 has length at least $(q+3)/2$, the (even lengthed) directed cycle in each $D_i$ has length at least $ (q+3)/2$.  This implies $|V(D_i)|\ge (q+3)/2$ for all $i$, and that the length of the cycle in each $D_i$ is at least 4 (since this cycle length is an even integer that is at least $(q+3)/2\ge 3$).  We can thus apply \Cref{prop:unicyclic} to obtain $q$-kernels $Q_i$ for each $D_i$ with $|Q_i|\le \frac{2}{q+3}|V(D_i)|$ and $Q_i\sub U_i$.
		
		Let $Q=\bigcup Q_i$.  Because $Q_i\sub  U_i\sub  K$ for all $i$, we have $Q\sub K$ (and hence is an independent set).  We also have $|Q|\le \frac{2}{q+3}|V(D)|$ (since $|Q_i|\le \frac{2}{q+3}|V(D_i)|$ for all $i$)  and that $\dist_D(Q,x)\le q$ for all $x\in V(D)$ (since $x\in V(D_i)$ for some $i$ and $\dist_D(Q,x)\le \dist_{D_i}(Q_i,x)\le q$).  We conclude that $Q$ is a $q$-kernel of the desired size of $D'\sub D$ and hence of $D$, completing the proof.
	\end{proof}
	We now use this to complete the proof of our main result.
	\begin{proof}[Proof of \Cref{thm:bipartite}]
		Recall that we wish to show that if $D$ is a source-free digraph which has a kernel $K$ and if $q,\ell\ge3$ are integers such that $\ell\le (q+3)/2$ and such that every directed cycle of $D$ of even length has length at least $\ell$, then $D$ contains a $q$-kernel $Q$ of size at most $\frac{1}{\ell}|V(D)|$ with $Q\sub K$.
		
		By the hypothesis of the theorem, we can apply \Cref{thm:bipartiteSpecial} with $q'=2\ell-3\ge 3$ to obtain a $q'$-kernel $Q$ of $D$ with size at most $\frac{2}{q'+3}|V(D)|=\frac{1}{\ell}|V(D)|$ that lies in $K$.  Note that $Q$ is also a $q$-kernel since $q\ge q'=2\ell-3$ by hypothesis, proving the result.
	\end{proof}
	One can extend \Cref{thm:bipartite} somewhat with some (non-trivial) extra work.  To this end, given a digraph $D$, we say that a partition $V_1\cup \cdots V_r$ of $V(D)$ is an \textit{$r$-cyclic partition} if every arc $uv\in E(D)$ has $u\in V_i$ and $v\in V_{i+1}$ for some $1\le i\le r$ (where here and throughout we write our indices mod $r$).  By adapting our methods we claim without proof that we can establish the following generalization of \Cref{thm:bipartite}.
	\begin{thm}\label{thm:rCyclic}
		Let $D$ be a source-free bipartite digraph with $r$-cyclic partition $V_1\cup \cdots \cup V_r$.  If $q\ge r^2-1$ and $\ell\ge (r^2+r)/2$ are integers such that $\ell\le (q+1+r)/2$, $\ell\equiv 0\mod r$, and such that every directed cycle of $D$ has length at least $\ell$, then $D$ contains a $q$-kernel $Q$ of size at most $\frac{1}{\ell}|V(D)|$ with $Q\sub V_1$.
	\end{thm}

	As \Cref{thm:rCyclic} is not our focus, we defer  discussing the tightness of its parameters and the details of its proof to the Appendix.
	
\section{Concluding Remarks}\label{sec:con}
	In this section we outline a number of remaining problems involving $q$-kernels.
	
	\subsection{Disjoint $q$-kernels}
	
	Recall that \Cref{thm:disjointWeak} guarantees the existence of $r$ disjoint $q$-kernels with $q\le 2^{r+1}$ provided $D$ contains no $(r-1)$-source sets.  The best construction we know of is $D=C_{2r-1}$, which requires $q\ge 2r-2$  in order to have $r$-disjoint $q$-kernels. We suspect that this lower bound is much closer to the truth compared to the upper bound.
	\begin{conj}\label{conj:disjoint}
		There exists $\ep,r_0>0$ such that if $D$ contains no $(r-1)$-source sets with $r\ge r_0$, then $D$ contains $r$-disjoint $q$-kernels with $q\le (2-\ep)^r$.
	\end{conj}
	The main barrier for proving \Cref{conj:disjoint} is \Cref{prop:noStructures}, which roughly involves replacing $D$ by its ``square'' and is the primary source of the growth rate of $2^r$.  As \Cref{prop:noStructures} is the heart of our present argument, significant new ideas would be needed in order to prove \Cref{conj:disjoint}.  

	\subsection{Small $q$-kernels}
	We next turn to questions involving small $q$-kernels in source-free digraphs with certain restrictions.
	\subsubsection{Restricted Cycle Lengths}
	\Cref{quest:smallCycles} involving digraphs with restricted cycle lengths seems too difficult to solve in general, but many interesting problems remain even in the case when $D$ is bipartite.  For example, \Cref{thm:bipartite} gives optimal bounds when $D$ is a bipartite digraph of an appropriate girth.  It would be interesting to extend this result to all bipartite $D$.  One promising approach for odd $q$ would be the following, where here we recall a digraph is unicyclic if every vertex has in-degree 1 and if the underlying graph is connected. 
	
	\begin{quest}\label{quest:unicyclic}
		Let $D$ be a unicyclic digraph with directed cycle of length $2\ell$ and bipartition $U\cup V$.  If $q\ge 3$ is odd, do there exist $q$-kernels $Q_U\sub U$ and $Q_V\sub V$ such that
		\[|Q_U|+|Q_V|\le 2\cdot\frac{\lceil \ell/(q+1) \rceil}{\ell}\cdot |V(D)|?\]
	\end{quest}
	As mentioned in the introduction, it is not true in general that we can always guarantee e.g.\ $|Q_U|\le \frac{\lceil \ell/(q+1) \rceil}{\ell}\cdot |V(D)|$, but \Cref{quest:unicyclic} suggests that on average at least one of the sets $Q_U$ or $Q_V$ will meet this bound if $q$ is odd.  If this were true, then one could use the same reduction from the proof of \Cref{thm:bipartite} to give a positive answer to \Cref{quest:smallCycles} whenever $q$ is odd and $D$ is bipartite (with the additional conclusion that $Q$ is contained in either $U$ or $V$, though we can not specify which set it will lie in ahead of time).
	
	The situation is more complex for bipartite $D$ when $q$ is even, as we can not in general find small $q$-kernels contained in either $U$ or $V$.  For $q=2$ we managed to prove the weak bound $|Q|<\half |V(D)|$ whenever $D$ (in particular) avoids directed 2-cycles and 4-cycles.   We suspect this can be improved substantially.
	\begin{conj}\label{conj:quasiGirth}
		There exists an $\ep>0$ such that if $D$ is a source-free bipartite digraph with no directed 2-cycles or 4-cycles, then $D$ contains a quasikernel $Q$ with $|Q|\le\left(\frac{1}{2}-\ep\right)|V(D)|$.
	\end{conj}
	The optimal possible value here would be $\ep=1/10$ due to $C_{10}$.  We note that \Cref{quest:smallCycles} predicts that \Cref{conj:quasiGirth} should hold even without the hypothesis that $Q$ is bipartite, but this seems comparable in difficulty to the Small Quasikernel Conjecture itself.
	
	\subsubsection{Variants of Source-Free Digraphs}
	Source-free digraphs are exactly digraphs with minimum in-degree 1.  As such, the following is a natural followup question to the Small Quasikernel Conjecture.
	\begin{quest}\label{quest:minDegree}
		For integers $\delta \ge 1$ and $q\ge 2$, what is the smallest constant $c_{\delta,q}$ such that every digraph $D$ with minimum in-degree $\delta$ contains a $q$-kernel $Q$ with $|Q|\le c_{\delta,q}|V(D)|$?
	\end{quest}
	For example, the Small Quasikernel Conjecutre is equivalent to the statement $c_{1,2}=\half $.  The full extent of our knowledge around this question is as follows.
	\begin{prop}\label{prop:minDegree}
		Let $c_{\delta,q}$ be defined as above.
		\begin{itemize}
			\item[(a)] For all $\delta\ge 1$ and $q\ge 2$, we have $c_{\delta,q}\ge \frac{1}{\delta+1}$.
			\item[(b)] If $q\ge 2^{\delta+2}$, then $c_{\delta,q}= \frac{1}{\delta+1}$.
			\item[(c)] For all $\delta\ge 1$, we have $c_{\delta,2}\ge \frac{1}{2}$.
		\end{itemize}
	\end{prop}
	\begin{proof}
		Part (a) follows by considering a bidirected $K_{\delta+1}$, which has minimum in-degree $\delta$ and every non-empty independent set has size $\frac{1}{\delta+1}|V(D)|$.
		
		For (b), observe that any digraph $D$ with minimum in-degree $\delta$ has no $\delta$-source sets.  As such, \Cref{thm:disjointWeak} implies that $D$ contains $\delta+1$ disjoint $q$-kernels for any $q\ge 2^{\delta+2}$, and in particular one of these has size at most $\frac{1}{\delta+1}|V(D)|$.
		
		For (c), we will show that $c_{\del,2}\ge \frac{\ell^2}{(\ell+1)(2\ell+1)}$ for every integer $\ell\ge 1$, from which the result follows.  Let $T$ be an Eulerian tournament on $v_1,\ldots,v_{2\ell+1}$ and let $k= \del \ell$ (though any large integer will also work).  We define a digraph $D$ as follows.  The vertices of $D$ lie in disjoint vertex sets $V_i,V_i'$ with $|V_i|=\del$ and $|V_i'|=k$ for all $1\le i\le 2\ell+1$.  We add every arc from $V_i$ to $V_j$ if $v_i\to v_j$ is an arc in $T$, and we add every arc from $V_i$ to $V_i'$ for all $i$.   In other words, $D$ is formed by taking $T$, adding leaves $v'_i$ to each vertex, and then blowing up the $v_i$ and $v_i'$ vertices into independent sets of size $\del$ and $k$, respectively.
		
		Observe that $D$ has minimum in-degree $\del$. We claim that every quasikernel $Q$ of $D$ contains at least $k\ell$ vertices.  Indeed, because $Q$ is independent, it can contain vertices from at most one $V_i$ set.  It follows that $Q$ must contain every vertex of $V_j'$ for each $j$ with $v_j\notin N^+_T[v_i]$ (as these vertices will not be reached by a path of length at most 2 otherwise).  There are exactly $\ell$ such vertices $v_j$, from which the claim follows.  Since $|V(D)|=(k+\del)(2\ell+1)$, this claim implies that $c_{\del,2}$ must satisfy 
		\[c_{\del,2}\ge \frac{k\ell}{(k+\del)(2\ell+1)}=\frac{\ell^2}{(\ell+1)(2\ell+1)},\]
		proving the result.
	\end{proof}
	The above suggests the following weak version of the Small Quasikernel Conjecture, which is asymptotically optimal in view of \Cref{prop:minDegree}(c).
	\begin{conj}
		There exists some $\delta\ge 1$ such that every digraph $D$ of minimum in-degree at least $\delta$ contains a quasikernel $Q$ with $|Q|\le \frac{1}{2}|V(D)|$.
	\end{conj}
	The proof of \Cref{prop:strongLargeSmall} can be adjusted to continue to hold if we replace the source-free hypothesis with in-degree at least $\del$ for any fixed $\del$ (which is in fact how we arrived at the construction for \Cref{prop:minDegree}(c)), so weak versions of the Large Quasikernel Conjecture continue to be an obstacle for this conjecture as well.  It is unclear what the true behavior of $c_{\delta,q}$ should be in general, especially when $\delta$ is large in terms of $q\ge 3$.

	Another variant of being source-free is being strongly connected.  In \cite[Example 17]{erdHos2023small} it is shown that there exist strongly connected digraphs with smallest quasikernels of size $(\half -o(1))|V(D)|$.  For $q\ge 3$ we ask the following.
	\begin{quest}
		If $D$ is a strongly connected digraph and $q\ge 3$ is an integer, does there exist a $q$-kernel $Q$ of $D$ such that $|Q|\le \frac{|V(D)|}{q+1}+O_q(1)$?
	\end{quest}
	This bound would be best possible by considering a collection of $C_{q+2}$'s sharing a common vertex.  We have not spent much time thinking about this problem and it is likely that better constructions exist, especially in view of \cite[Example 17]{erdHos2023small} for quasikernels.  Still, it would be interesting to see if any non-trivial bounds could be obtained under the assumption of being strongly connected.  The best we can show is the existence of $q$-kernels of size at most roughly  $\frac{|V(D)|}{\log q}$ by using \Cref{thm:disjointWeak} to show the existence of around $\log q$ disjoint $q$-kernels.
	
	
	\subsection{Large Quasikernels}
	
	We asked in \Cref{conj:largeQuasi} whether every digraph contains a ``large'' quasikernel, namely one with $|N^+[Q]|\ge \half |V(D)|$, and we made some modest progress towards this goal in \Cref{thm:large} by proving $|N^+[Q]|\ge V(D)|^{1/3}$.  We believe that our bounds can be improved by proving the following result.
	\begin{conj}\label{conj:acyclic}
		Every digraph $D$ either contains a subset $A\sub V(D)$ with $D[A]$ acyclic and $|A|\ge |V(D)|^{1/2}$, or a vertex $y\in V(D)$ which is contained in some quasikernel and which has $\deg^+(y)\ge |V(D)|^{1/2}-1$.
	\end{conj}
	The bounds of \Cref{conj:acyclic} are best possible, as can be seen by considering the disjoint union of Eulerian tournaments of size $|V(D)|^{1/2}$.  We note that our proof of \Cref{thm:large}(a) essentially reduced to showing the conjecture above holds with $1/2$ replaced by $1/3$, and in particular \Cref{conj:acyclic} would improve \Cref{thm:large}(a) to the stronger bound $|N^+[Q]|\ge |V(D)|^{1/2}$, which is the natural barrier for our present approach.

	There are various weakenings of the Large Quasikernel Conjecture that one could try to solve, such as by improving the bounds for $q$-kernels, or by replacing $|N^+[Q]|$ with the larger quantity $\sum_{v\in Q} |N^+[v]|$.  Another direction would be to look for stronger versions of \Cref{conj:largeQuasi}, such as the following.
	
	\begin{quest}\label{quest:EvenLarger}
		Does every digraph $D$ contains a quasikernel $Q$ with $|N^+(Q)|\ge \half |V(D)\sm Q|$?  Equivalently, does there exist a quasikernel with $|N^+[Q]|\ge \frac{|V(D)|+|Q|}{2}$?
	\end{quest}
	This result would be sharp by considering the disjoint union of Eulerian tournaments of arbitrary sizes (cf the Large Quasikernel Conjecture which is only asymptotically sharp).  It would be interesting to know if, like the Large Quasikernel Conjecture, \Cref{quest:EvenLarger} is a consequence of the Small Quasikernel Conjecture.
	
	\textit{Update:} shortly after posting our paper, it was shown by Ai, Liu, and Peng \cite{ai2024variable} that \Cref{quest:EvenLarger} is indeed a consequence of the Small Quasikernel Conjecture.  Moreover, a partial converse to \Cref{prop:strongLargeSmall} was established, showing that proving non-trivial linear bounds for the Large Quasikernel Conjecture is equivalent to proving non-trivial linear bounds for the Small Quasikernel Conjecture.

	\textbf{Acknowledgments}.  We thank Ruth Luo for introducing us to the Small Quasikernel Conjecture, as well as Xiaoyu He, Noah Kravitz, and Bhargav Narayanan for fruitful discussions.  We further thank Noah Kravitz for correcting typos from an earlier draft, as well as for shortening our original proof of \Cref{lem:disjointUnion}, and we thank P\'eter L.\ Erd\H{o}s for pointing out that our previous statement of \Cref{prop:smallQuasiBipartite} could be strengthened to its current form with an identical proof.
	
	
	\bibliographystyle{abbrv}
	\bibliography{refs}
	 
	\newpage 
	\appendix
	\section{Strengthening the Results}\label{append}
	
	Here we discuss a few technical ideas themed around improving the main results of the paper.

	\subsection{Variants of the Large Quasikernel Conjecture}
	
	Here we briefly discuss some variants of the Large Quasikernel Conjecture which are natural to consider but which ultimately turn out to be non-interesting.
	
	First, one could form an (a priori) stronger version of the Large Quasikernel Conjecture by considering digraphs $D$ together with a weight function $w:V(D)\to \R_{\ge 0}$.  In this setting, a natural conjecture would be that for every weighted digraph $D$, there exists a qusikernel $Q$ such that the weight of the vertices of $N^+[Q]$ is at least half of the total weight on $V(D)$.  This stronger conjecture is implicitly solved for tournaments in \cite[Theorem 12]{ai2023results}, and we note that the proof for \Cref{prop:small2Large} can be used to show that this weighted conjecture is also implied by the Small Quasikernel Conjecture (essentially by giving each vertex about $k\cdot w(u)$ new neighbors).  
	
	However, this weighted version turns out to be implied by the unweighted version.  Somewhat more precisely, if $w$ is integral valued, then we can replace each $u\in V(D)$ by a set $U$ of size $w(u)$ and each arc $uv$ by including every possible arc from $U$ to $V$, in which case a solution to the Large Quasikernel Conjecture in this auxiliary digraph lifts to a solution for the weighted version.  This solves the problem for $w$ integral valued, which solves it for rational values, which can then be used to solve it for real values by approximating by rationals.  As such, there is no loss in generality by only considering \Cref{conj:largeQuasi}.
	
	Second, one might consider a variant of the Large Quasikernel Conjecture for $q$-kernels by replacing $N^+[Q]$ by the set of all vertices that are at distance at most $q-1$ from $Q$.  However, if $q\ge 3$ then we can take $Q$ to be any quasikernel (which is in particular a $q$-kernel), in which case this set ends up being all of $V(D)$.
	
	\subsection{Improving the Disjointness Result}\label{sec:modest}
	
	Here we discuss the problem of improving the bounds of \Cref{thm:disjointFull} for finding $\beta$-kernels in digraphs without source-set.  As noted around \Cref{conj:disjoint}, the largest loss in our argument comes from \Cref{prop:noStructures} (where we essentially take the square of $D$), and this step seems difficult to get around with our present approach.  However, there is one (somewhat technical) place where our argument can likely be improved to give a moderate strengthening of our results.
	
	Specifically, the loss we consider comes from \Cref{lem:D'} where we take each proper $(r-2)$-pseudo-source set $S$ and add all possible arcs within $N^-(S)$.  With this process, we destroy all $(r-2)$-pseudo-source sets of $D$ at the cost of decreasing the path lengths in $D'$ by as much as $2(r-2)-1$ compared to that of $D$.  It seems plausible that one can add fewer arcs to $S\cup N^-(S)$ while maintaining the $(r-2)$-pseudo-source set condition and decreasing the path length by significantly less.   Unwinding our proof shows that it would suffice\footnote{Specifically, given a digraph $D$, we take each strongly connected component $C$ which contains an $(r-2)$-pseudo-source set and replace $D[C]$ with $D'[C]$ as defined in \Cref{prob:paths}.  We can apply this result since $C$ contains at least $r$ vertices by \Cref{cl:leastR} and \Cref{cl:scc} (with this second claim implying that $C=S\cup N^-(S)$), and the condition $|V(D)|-\al(D)\le r-2$ holds since $|S|\le r-2$ and $|N^-(S)|\le \al(D)$.  The proof of \cref{cl:mostOne} essentially says that the set of vertices of $D'[C]$ which is not in some $(r-2)$-pseudo-source set of $D'$ is an independent set of $D'$, i.e.\ has size at most $\al(D')$, so $|V(D')|-\al(D')\ge r-1$ implies there are no $(r-2)$-pseudo-source sets in $D'$.} to solve the following (seemingly tractable) problem, where here $\al(D)$ denotes the size of a largest independent set of $D$.

	\begin{prob}\label{prob:paths}
		For each $r\ge 3$, determine the smallest number   $\Del(r)$ such that the following holds:  If $D$ is a strongly connected digraph on at least $r$ vertices such that $|V(D)|-\al(D)\le r-2$, then there exist a superdigraph $D'\supseteq D$ on the same vertex set $V(D)$ such that $|V(D')|-\al(D')\ge r-1$ and such that for all $u,v\in V(D')$, we have
		\[d_{D'}(u,v)\ge d_D(u,v)-\Del(r).\]
	\end{prob}
	
	Our current argument essentially shows that we can take $\Del(r)=2(r-2)-1$ (via adding all possible arcs within $D$).  We believe one can get this to work with $\Del(r)=r-3$, which would be best possible by considering a directed $r$-cycle.   If one could prove this for $\Del(r)=r-3$, then one could essentially replace every $2(r-2)$ in our proof with $r-2$.  In particular, the analog of \Cref{thm:disjointFull} would guarantee $\al$-kernels in $(r-1)$-source-free digraphs with $\al_r=r-1$, and such a result would be best possible by considering the directed $r$-cycle.

	\subsection{Difficulties with Bipartite Digraphs}
	In \Cref{thm:bipartite} we managed to verify some special cases of \Cref{quest:smallCycles} by finding small $q$-kernels whenever $D$ has a bipartition $U\cup V$ and avoids certain cycle lengths.  Moreover, \Cref{thm:bipartite} can be used to guarantee that these small $q$-kernels lie entirely within, say, $U$.  This additional feature of having $Q\sub U$ was vital to our proof of \Cref{thm:bipartiteSpecial}, as otherwise we would not be able to guarantee that the union of the $Q_i$ sets from our reduced digraph $D'$ would still be independent.  However, the following result shows that it is not possible in general to solve \Cref{quest:smallCycles} for bipartite digraphs if we continue to require  $Q\sub U$.
	\begin{lem}\label{lem:badBipartite}
		For all $q,\ell\ge 1$, there exists a digraph $D$ with bipartition $U\cup V$ such that every directed cycle of $D$ has length $2\ell$ and such that any $q$-kernel $Q$ of $D$ which lies in $U$ has $|Q|\ge \frac{1}{q} |V(D)|$.
	\end{lem}
	For example, when $q=3$ this shows that regardless of the cycle conditions placed on $D$, it is impossible to improve upon the bound $|Q|\le \frac{1}{3}|V(D)|$ obtained from \Cref{thm:bipartite} if we continue insisting that $Q\sub U$.  As such, something like \Cref{quest:unicyclic} would be needed to solve \Cref{quest:smallCycles} for bipartite digraphs when $q=3$.
	\begin{proof}
		Let $D$ be the digraph obtained from a directed cycle $u_1,v_1,\ldots,u_\ell,v_\ell$ after attaching directed paths of length $q-2$ to each $v_i$ vertex.  If $u_i\in U$ for all $i$, then it is not difficult to see that every $q$-kernel contained in $U$ must have size at least $\ell$ (namely, $Q$ must contain either $u_i$ or at least one vertex in the path attached to $v_i$ for all $i$), proving the result.
	\end{proof}
	There exist some constructions that are even better than those of \Cref{lem:badBipartite}.  For example, for $q=5$ and $\ell=3$ one can consider a directed 6-cycle after attaching a leaf to each vertex of $V$.  In this case the smallest 5-kernel contained in $U$ satisfies $|Q|=\frac{2}{9}|V(D)|$, which is a larger fraction than is guaranteed by \Cref{lem:badBipartite} and which (unlike \Cref{lem:badBipartite} in this case) shows that \Cref{quest:smallCycles} can not hold when $q=5$ and $L(D)=\{\ell\}$ if we insist on $Q\sub U$.

	\subsection{Beyond Bipartite Digraphs}\label{sec:beyond}
	Recall that we say a partition $V_1\cup \cdots V_r$ of $V(D)$ is an \textit{$r$-cyclic partition} if every arc $uv\in E(D)$ has $u\in V_i$ and $v\in V_{i+1}$ for some $1\le i\le r$ (where here and throughout we write our indices mod $r$).  Here we discuss \Cref{thm:rCyclic}, which we restate below for convenience.
	\begin{thm*}
		Let $D$ be a source-free bipartite digraph with $r$-cyclic partition $V_1\cup \cdots \cup V_r$.  If $q\ge r^2-1$ and $\ell\ge (r^2+r)/2$ are integers such that $\ell\le (q+1+r)/2$, $\ell\equiv 0\mod r$, and such that every directed cycle of $D$ has length at least $\ell$, then $D$ contains a $q$-kernel $Q$ of size at most $\frac{1}{\ell}|V(D)|$ with $Q\sub V_1$.
	\end{thm*}
		Note that\Cref{thm:rCyclic} requires the extra hypothesis $\ell\equiv 0\mod r$ compared to \Cref{thm:bipartite}, but this is somewhat superfluous since every cycle length of $D$ must be a multiple of $r$.  \Cref{thm:rCyclic} can be tight by considering disjoint unions of $C_\ell$, as well as by considering disjoint unions of $C_{2\ell}$ whenever $2\ell\ge q+2$ (which always occurs at the maximum possible value $\ell=\lfloor (q+1+r)/2\rfloor$).

	The bound $q\ge r^2-1$ is necessary for \Cref{thm:rCyclic} to hold: if we had $q=r^2-r-1$, then one could take $D$ to be $C_{q+1}$ after attaching a leaf to each of the $(q+1)/r=r$ vertices in $V_r$.  In this case, any $q$-kernel contained in $V_1$ must contain at least two vertices, but the bound we ask for is $\lfloor\frac{2|V(D)|}{q+1+r} \rfloor=\lfloor\frac{2(r^2-1)}{r^2} \rfloor=1$.
	
	Similar to our proof of \Cref{thm:bipartiteSpecial}, \Cref{thm:rCyclic} will follow from the following special case of $\ell=(q+1+r)/2$ and $q\equiv -1\mod r$.
	\begin{thm}\label{thm:rCyclicSpecial}
		Let $D$ be a source-free digraph with $r$-cyclic partition $V_1\cup \cdots \cup V_r$.  If $q\ge r^2-1$ is an integer with $q\equiv -1\mod r$ and is such that every directed cycle of $D$ has length at least $(q+1+r)/2$, then there exists a $q$-kernel $Q$ of $D$ such that $|Q|\le \frac{2}{q+1+r} |V(D)|$ with $Q\sub V_1$.
	\end{thm}
	With this result we can prove \Cref{thm:rCyclic} by applying \Cref{thm:rCyclicSpecial} with $q':=2\ell-r-1\ge r^2-1$, which yields a $q'$-kernel of size at most $\frac{1}{\ell}|V(D)|$ that is also a $q$-kernel since $q\ge q'$ by hypothesis.  We now discuss the ideas of this proof.
	
	\begin{proof}[Sketch of Proof of \Cref{thm:rCyclicSpecial}]
		As before, it suffices to prove the result when $D$ is unicyclic with cycle length $r\ell\ge 2r$ and $|V(D)|\ge (q+1+r)/2$.  We label the vertices of the cycle by $v_1^1,v_1^2,\ldots,v_1^r,v_2^1,v_2^2$, and so on.  We say a vertex $w$ has type $j$ if it is closest to $\{v_j^1,\ldots,v_j^r\}$ and let $\ty(j)$ denote the number of vertices of type $j$.  Essentially the same proof as before shows for all $j$ that
		\[\ty(j)<(q+1+r)/2,\]
		though a little extra care is needed to deal with the case when $(q+1+r)/2$ is not an integer.
		
		Let $m$ be such that either $q=-1-r+2rm$ or $q=-1-2r+2rm$ and let $a$ be such that $\ell\equiv a \mod m$.  The case where  $\frac{2}{q+1+r}|V(D)|<\lceil \ell/m\rceil$ is very similar to the $r=2$ case: here we can follow the same logic as in our proof of \eqref{eq:smallD} to show
		\[|V(D)|\le r\ell+q-rm-ra+r,\]
		where again we have that $q-rm-ra+r>0$ by the same reasoning as before.  From this it is easy to show that $\ty(j)\le q-rm-ra+2r+1$ and that $\min_j \ty(j)\le q-rm-ra+2r+1$.  Assuming $\ty(\ell)=\min_j\ty(j)$, we can take $Q=\{v_{1+mt}^1:0\le t\le \ell/m-1\}$ which will have the desired size.  If $w\in V(D)$ is of type $j\ne \ell$ then we have
		\[\dist(Q,w)\le r(m+a-2)+\ty(j)-1\le q,\]
		and a similar argument holds for $j=\ell$ showing that $Q$ is indeed a small $q$-kernel.
		
		It thus remains to deal with the case $\frac{2}{q+1+r}|V(D)|\ge \lceil \ell/m\rceil$, and it is here that things  get more complicated.  Naively, we would like to use the set $Q=\{v_{1+mt}^1:0\le t\le \lceil \ell/m\rceil-1\}$ which has the desired size, but this turns out\footnote{Specifically this will be a $q$-kernel if we always have $\ty(j)\le q-rm+r+1$, and we are only guaranteed $\ty(j)<(q+1+r)/2$.  If $q=-1-r+2rm$ then we have $\ty(j)<rm=q-rm+r+1$, so this is fine.  However, if $q=-1-2r+2rm$ we only have $\ty(j)<rm-r/2$, and this will only imply $\ty(j)\le q-rm+r+1=rm-r$ if $r=2$.} to be a $q$-kernel in general only if either $q=-1-r+2rm$ or if $r=2$.  However, we can still get the idea of picking every $1/m$ vertices to work when $q=-1-2r+2rm$ and $\frac{2}{q+1+r}|V(D)|\ge \lceil \ell/m\rceil$ provided we add in a few extra vertices.
		
		To this end, define $B$ to be the set of $j$ with $\ty(j)\ge rm-r+1$, and let $B'\sub B$ be the elements with $j=mt$ for some integer $1\le t\le \ell/m$ (noting that the $j\in B'$ are exactly the $j$ such that we can not guarantee $\dist(Q,w)\le q$ in the argument above).  Possibly by reindexing our cycle, we can assume that $|B'|\le  \lfloor |B|/m\rfloor $ (since, for example, randomly choosing the vertex labeled $v_1^1$ gives $\E[|B'|]\le |B|/m$).  Define
		\[Q=\{v_{1+mt}^1:0\le t\le \lceil \ell/m\rceil-1\}\cup \{v_j^1:j\in B'\}.\]
		
		It is not difficult to show that $Q$ is a $q$-kernel.  Since $|Q|\le  \lceil \ell/m\rceil+\lfloor |B|/m\rfloor$, it suffices to show
		\begin{equation}\frac{2}{rm-r} |V(D)|=\frac{2}{q+1+r}|V(D)|\ge \lceil \ell/m\rceil+\lfloor |B|/m\rfloor.\label{eq:lastQ}\end{equation}
		Since we assume $\frac{2}{q+1+r}|V(D)|\ge \lceil \ell/m\rceil$, there is nothing to prove if $|B|<m$, so we can assume $|B|\ge m$.   We can also assume $m\ge 3$, as otherwise $q\ge r^2-1$ does not hold for $r\ge 3$ (since $q=-1-2r+2rm$).
		
		By definition of $B$ we have 
		\[|V(D)|\ge r(\ell-|B|)+(rm-r+1)|B|=r \ell+(1-2r+rm)|B|.\]
		Multiplying both sides by $\frac{2}{rm-r}$ and using some crude estimates gives
		\[\frac{2}{rm-r} |V(D)|\ge  \frac{2}{m}\ell+\frac{2(m-2)}{m-1}|B|=\frac{2}{m}\ell+ \frac{2m-5}{m-1}|B|+\frac{1}{m-1}|B|\ge \frac{2}{m}\ell +1+\lfloor |B|/m\rfloor,\]
		where this last step used $|B|\ge m-1$ and that $2m-5\ge 1$ since $m\ge 3$.  Observe that $ \frac{2}{m}\ell +1\ge \lceil \ell/m\rceil$ (this is trivial if $\ell\le m$, otherwise $\lceil \ell/m\rceil\le 2\ell/m$).  We thus conclude that \eqref{eq:lastQ} holds, proving the result.

	\end{proof}

\end{document}